\documentclass[reqno,12pt,a4paper]{amsart}

\usepackage[utf8]{inputenc}

\usepackage{enumerate}
\usepackage{amstext,amsmath,amsthm,amsfonts,amssymb,amscd}
\usepackage{latexsym,mathrsfs,dsfont,euscript}
\usepackage{pict2e}
\usepackage{color}
\usepackage{marginnote}

\usepackage{hyperref} 
\hypersetup{
  colorlinks,%
  citecolor=blue,%
    filecolor=red,%
    linkcolor=red,%
    urlcolor=blue
}

\newtheorem{theorem}{Theorem}[section]
\newtheorem{proposition}[theorem]{Proposition}
\newtheorem{lemma}[theorem]{Lemma}
\newtheorem{corollary}[theorem]{Corollary}
\theoremstyle{definition}
\newtheorem{definition}{Definition}[section]

\theoremstyle{remark}

\numberwithin{equation}{section}

\DeclareMathOperator{\supp}{supp}
\newcommand{\abs}[1]{\left\vert#1\right\vert}

\newcommand{\proin}[2]{\left<#1,#2\right>}
\newcommand{\norm}[1]{\left\Vert#1\right\Vert}

\allowdisplaybreaks
\begin{document}
\title[]{On generalized divergence and Laplace operators as a matter of division of distributions}
%


\author[]{Hugo Aimar}
\email{haimar@santafe-conicet.gov.ar}
\author[]{Ivana G\'{o}mez}
\email{ivanagomez@santafe-conicet.gov.ar}
\thanks{This work was supported by the Ministerio de Ciencia, Tecnolog\'ia e Innovaci\'on-MINCYT in Argentina: CONICET and ANPCyT; and the UNL}
\subjclass[2010]{Primary. 35J05, 26A33, 46F12}


\keywords{Laplace Operator, Fractional Laplacian, Discrete Laplacians, Generalized Divergence}

\begin{abstract}
Starting from the approach to the Laplacian with respect to coupling measures and undirected weighted graphs, we provide a setting for a general point of view for a Kirchhoff type divergence and a Laplace operators built on the trivial gradient of order zero $f(y)-f(x)$. We consider some particular classical and new instances of this approach.
\end{abstract}
\maketitle


\section{Introduction}\label{sec:intro}

We are going to profusely illustrate the problem that can be better stated if we start from the general setting. Of course the roots are, as in \cite{AiGoKirchhoff1}, in the definition of divergence and Laplacian on graphs. See \cite{Bronstein}.

Let $X$ be a set. Let $\mathscr{S}_1$ be a topological algebra of real valued functions defined on $X$. Let $\mathscr{S}_2$ be a topological algebra of real functions defined on $X\times X$. Set $\mathscr{S}_1\otimes \mathscr{S}_1$ to denote the space of the tensor products $(\varphi\otimes\eta)(x,y)=\varphi(x)\eta(y)$ with $\varphi$ and $\eta$ both in $\mathscr{S}_1$. Let $\sigma(\mathscr{S}_1\otimes \mathscr{S}_1)$ denote the linear span of $\mathscr{S}_1\otimes \mathscr{S}_1$. Assume that $\mathscr{S}_1\otimes \mathscr{S}_1$ is continuously contained in  $\mathscr{S}_2$ and also the density in $\mathscr{S}_2$ of $\sigma(\mathscr{S}_1\otimes \mathscr{S}_1)$. In other words
\begin{equation}\label{eq:S2definition}
\overline{\sigma(\mathscr{S}_1\otimes \mathscr{S}_1)} =\mathscr{S}_2,
\end{equation}
where the closure is, of course, taken in the topology of $\mathscr{S}_2$.

Set $\mathscr{S}_i^{'}$ to denote the topological dual space of $\mathscr{S}_i$; $i=1,2$. We shall use single brackets $\proin{}{}$ to denote the duality $\mathscr{S}_1$, $\mathscr{S}_1^{'}$ and double brackets $\langle\!\langle , \rangle\!\rangle$ to denote the duality $\mathscr{S}_2$, $\mathscr{S}_2^{'}$.
\begin{lemma}\label{lem:FunctionalSigmaonDistributions}
Let $S\in\mathscr{S}_2^{'}$ and $\Phi\in\mathscr{S}_2$. Then the functional $\Sigma$ defined for $\varphi\in\mathscr{S}_1$ by
\begin{equation*}
\proin{\Sigma}{\varphi}=\langle\!\langle S,\varphi\Phi \rangle\!\rangle
\end{equation*}
is well defined and belongs to $\mathscr{S}_1^{'}$. Sometimes we write $\Sigma_{S,\Phi}$ to recall the dependence on $S$ and $\Phi$
\end{lemma}
\begin{proof}
Notice first that from \eqref{eq:S2definition}, for $\Phi\in\mathscr{S}_2$ we have that the function $(\varphi\Phi)(x,y)=\varphi(x)\Phi(x,y)$ also belongs to $\mathscr{S}_2$. Hence $\Sigma$ is well defined and linear on $\mathscr{S}_1$.The continuity follows from the continuity of the inclusion of $\mathscr{S}_1\otimes\mathscr{S}_1$ in $\mathscr{S}_2$.
\end{proof}

The division of distributions is in general impossible. Sometimes it makes sense and it is possible to obtain a quotient. When $S\in\mathscr{S}_2^{'}$, $\Phi\in\mathscr{S}_2$ and $T\in\mathscr{S}_1^{'}$ are given and it is possible the division of $\Sigma=\Sigma_{S,\Phi}$ by $T$, we are in position to define the Kirchhoff divergence of $\Phi$ with respect to $S$ and $T$. Actually and formally
\begin{equation*}
Kir_{T,S}\Phi=\frac{\Sigma_{S,\Phi}}{T}.
\end{equation*}
Let us precise the above. Given $T\in\mathscr{S}_1^{'}$, $S\in\mathscr{S}_2^{'}$ and $\Phi\in\mathscr{S}_2$, a function $\psi: X\to\mathbb{R}$ is said to be a Kirchhoff divergence of $\Phi$ with respect to $T$ and $S$ if
\begin{itemize}
	\item[\textit{(1.2.a)}] $\varphi \psi\in \mathscr{S}_1$ for every $\varphi\in\mathscr{S}_1$;
		and
	\item[\textit{(1.2.b)}] $\proin{T}{\varphi\psi}=\langle\!\langle S,\varphi\Phi \rangle\!\rangle$ for every $\varphi\in\mathscr{S}_1$.
\end{itemize}
Notice that \textit{(1.2.b)} is equivalent to $\psi T= \Sigma_{S,\Phi}$.

Any $\psi$ as before is denoted by $Kir_{S,T}\Phi$ or, when $S$ and $T$ are understood, by $Kir\,\Phi$. If $\Phi(x,y)$ can be taken to be $f(y)-f(x)$ for some $f:X\to\mathbb{R}$ we call the divergence of such $\Phi$, the Laplacian of $f$. Briefly
\begin{equation*}
\Delta_{S,T}f = Kir_{S,T} (f(y)-f(x)).
\end{equation*}
The above formally stated formula for the Kirchhoff divergence as a quotient $Kir_{T,S}\Phi = \frac{\Sigma_{S,\Phi}}{T}$ also applies to the induced Laplacian $\Delta_{T,S}f$. Briefly
\begin{equation*}
\Delta_{T,S}f= \frac{\Sigma_{S,\nabla f}}{T}
\end{equation*}
with $\nabla f(x,y)=f(y)-f(x)$.

Even more reckless than the division of distributions, but clearly related to it, is the idea of differentiation. If $T_k$ is a sequence of distributions in $\mathscr{S}_1^{'}$ that tends to zero and $S_k$ is a sequence of distributions in $\mathscr{S}_2^{'}$ that also tends to zero, we may ask for fixed $f$, for the existence of the limit of $\Delta_{T_k, S_k}f = \frac{\Sigma_{S_k,\nabla f}}{T_k}$.

In this paper we aim to consider classical and new cases of the above described general setting. In particular, we obtain the divergence operators of Kirchhoff type associated to fractional powers of the classical Laplacian and their generalization to metric spaces. In Section~\ref{sec:DiracDeltasEuclideanSpaces} the discrete case is considered. In particular we introduce the finite difference setting for the Laplacian and for their fractional versions. In Section~\ref{sec:DiracDeltasMetricMeasureSpaces} we introduce the generalization of the examples considered in \S\ref{sec:DiracDeltasEuclideanSpaces}, to general metric measure spaces. The only additional hypothesis, that allows the construction of dyadic families, is the finiteness of the metric or Assouad dimension of the space. As a special case we consider the space of Ahlfors type with index of regularity equal to one provided by the dyadic metric. Section~\ref{sec:GeneralMeasures} is devoted to the case in which the distributions $T$ and $S$ are provided by measures. Two special cases are considered, the first under an assumption of absolute continuity and the second provided by a deterministic type coupling of the involved measures. In Section~\ref{sec:SpositiveOrderDeterministicCoupling} we give an example with $S$ a distribution of positive order. Section~\ref{sec:ClassicalLaplacianRn} provides a Kirchhoff type operator which leads to the inclusion of the classical Laplacian in our general setting. In Section~\ref{sec:FractionalKirchhoffDivergencesEuclideanSpace} we study the Kirchhoff operators associated to the fractional powers of the Laplacian in the Euclidean space for the whole range $0<s<1$. Section~\ref{sec:FractionalKirchhoffdivergenceAhlforsSpaces} takes the problem of Section~\ref{sec:FractionalKirchhoffDivergencesEuclideanSpace} in Ahlfors regular spaces for $0<s<\tfrac{1}{2}$. In particular we consider the dyadic metric space and, using the spectral analysis in terms of Haar wavelets for the fractional Laplacian, we give a spectral formula for the Kirchhoff operator in this setting.  Section~\ref{sec:SomeExamples12bTpositiveOrder} is devoted to provide a Kirchhoff operator defined in terms of two distributions $T$ and $S$ of positive order. Actually $T$ and $S$ are singular integrals. The last section, \S\ref{sec:SomeConvergenceResults}, explores the problems of convergence of the
Kirchhoff operator given the convergence of the sequences of distributions $T$ and $S$. We define the derivative of $S$ with respect to $T$ and compute it for the finite difference cases introduced in \S\ref{sec:DiracDeltasEuclideanSpaces}. Some additional examples are provided in terms of coupling measures.

\section{Dirac deltas in Euclidean spaces}\label{sec:DiracDeltasEuclideanSpaces}

With the notation introduced in Section~\ref{sec:intro}, set $X=\mathbb{R}^n$, the $n$-dimensional Euclidean space. Let $\mathscr{S}_1=\mathscr{C}_c(\mathbb{R}^n)$ the space of compactly supported continuous real valued functions defined on $\mathbb{R}^n$. Let $\mathscr{S}_2=\mathscr{C}_c(\mathbb{R}^n\times \mathbb{R}^n)$ the continuous and compactly supported functions defined in $\mathbb{R}^n\times \mathbb{R}^n$. As usual for a given point $x_0\in\mathbb{R}^n$ we define $\delta_{x_0}$ as the unit mass measure at $x_0$. Or in notation of distributions $\proin{\delta_{x_0}}{\varphi}=\varphi(x_0)$ for every $\varphi\in\mathscr{S}_1$. Let $\{x_k: k\geq 1\}$ be a sequence of points in $\mathbb{R}^n$ and $\{a_k: k\geq 1\}$ a sequence of positive real numbers which is locally finite with respect to $\{x_k\}$. In other words, for every bounded set $B$ in $\mathbb{R}^n$ we have that $\sum_{\{k: x_k\in B\}}a_k<\infty$. Then $T=\sum_{k\geq 1}a_k\delta_{x_k}$ is a Borel measure in $\mathbb{R}^n$ which is finite on compact sets. Hence $T\in\mathscr{S}_1^{'}$. The distribution (measure) $T$ gathers the information of the nodes $\{x_k\}$ and their weights $\{a_k\}$. Let $\{w_{ij}: i,j\geq 1\}$ be a sequence of nonnegative real numbers which is locally finite with respect to the sequence of points in $\mathbb{R}^n\times\mathbb{R}^n$ given by $\{(x_i,x_j):i,j\geq 1\}$. Precisely for every bounded set $B$ in $\mathbb{R}^n\times \mathbb{R}^n$, $\sum_{\{(i,j):(x_i,x_j)\in B\}}w_{ij}<\infty$. Hence $S=\sum_{i,j} w_{ij}\delta_{(x_i,x_j)}\in\mathscr{S}_2^{'}$ and is actually a positive measure on the Borel sets of $\mathbb{R}^n\times \mathbb{R}^n$.

\begin{proposition}\label{propo:Kirfinite}
Let $T$ and $S$ be as above. Let $\Phi\in\mathscr{S}_2$. Then,
\begin{enumerate}[(a)]
\item the first marginal $S_\Phi^1$ of the measure $S_\Phi(A) = \iint_{\mathbb{R}^n\times \mathbb{R}^n}\Phi dS$ is absolutely continuous with respect to $T$;
\item  the Radon-Nikodym derivative of $S_\Phi$ with respect  to $T$ is the Kirchhoff divergence of $\Phi$ with respect to $S$ and $T$, $\frac{dS_{\Phi}^1}{dT}= Kir_{S,T}\Phi$;
\item $Kir_{S,T}\Phi(x_k)=\frac{1}{a_k}\sum_{j\geq 1}w_{kj}\Phi(x_k,x_j)$.
\end{enumerate}
\end{proposition}
\begin{proof}
The first marginal of $S_\Phi^1$ is given by
\begin{equation*}
S_\Phi^1=S_\Phi(E\times\mathbb{R}^n)=\iint_{E\times\mathbb{R}^n}\Phi dS =\sum_{\{k: x_k\in E\}}\sum_{j\geq 1} w_{ij}\Phi(x_k,x_j).
\end{equation*}
 Since each $a_k$ is positive, a subset $E$ of $\mathbb{R}^n$ has $T$ measure zero if and only if $E$ does not contain points of the sequence $\{x_k\}$. So that $S_\Phi^1(E)=0$, and $S_\Phi^1$ is absolutely continuous with respect to $T$. Hence from Radon-Nikodym theorem there  exists a function $\frac{dS_{\Phi}^1}{dT}$ such that $S_\Phi^1(E)=\int_E \frac{dS_{\Phi}^1}{dT} dT$ and $\int_{\mathbb{R}^n}\eta(x) dS_\Phi^1(x)=\int_{\mathbb{R}^n}\eta(x)\frac{dS_{\Phi}^1}{dT}(x) dT(x)$. Then
\begin{align*}
\left\langle\left\langle S, \varphi\Phi \right\rangle\right\rangle &= \iint_{\mathbb{R}^n\times\mathbb{R}^n}\varphi(x)\Phi(x,y)dS(x,y)\\
&= \sum_{k\geq 1}\sum_{j\geq 1}w_{kj}\varphi(x_k)\Phi(x_k,x_j)\\
&= \sum_{k\geq 1}\varphi(x_k)\left(\sum_{j\geq 1}w_{jk}\Phi(x_k,x_j)\right)\\
&= \int_{\mathbb{R}^n}\varphi(x) dS_\Phi^1(x)\\
&= \int_{\mathbb{R}^n}\varphi(x) \frac{dS_\Phi^1}{dT} dT(x)\\
&= \proin{T}{\varphi \frac{dS_\Phi^1}{dT}},
\end{align*}
for every $\varphi\in\mathscr{S}_1$. This proves \textit{(1.2.b)} in Section~\ref{sec:intro}. So that $Kir_{S,T}\Phi = \frac{dS_\Phi^1}{dT}$. On the other hand, if we write explicitly \textit{(1.2.b)} in this particular case we get that for every $\varphi\in\mathscr{S}_1$
\begin{equation*}
\sum_{k\geq 1}a_k \varphi(x_k)\varphi(x_j)=\sum_{k\geq 1}\varphi(x_k)\left(\sum_{j\geq 1}w_{jk}\Phi(x_k,x_j)\right).
\end{equation*}
Taking $\varphi\in\mathscr{S}_1$ such that $\varphi(x_k)=1$ and $\varphi(x_j)=0$ for $j\neq k$, we get $\textit{(c)}$,
\begin{equation*}
\psi(x_k)=Kir_{S,T}(x_k)=\frac{1}{a_k}\sum_{j\geq 1} w_{kj}\Phi(x_kx_j).
\end{equation*}
\end{proof}

The current hypothesis in the data sequences $\{x_k\}$, $\{a_k\}$ and $\{w_{kj}\}$ do not directly allow to take $\Phi(x,y)=f(y)-f(x)$. Not even for $f\in\mathscr{S}_1$. Hence, even when the Kirchhoff divergence type operator is well defined on $\mathscr{S}_2$, we can not directly define the Laplace operator on $\mathscr{S}_1$. If the measure $S$ is finite, then the above result allows taking $\Phi(x,y)=f(y)-f(x)$ for $f\in\mathscr{S}_1$ since, even when $\Phi$ does not belong to $\mathscr{S}_2$, the distribution $S$ extends naturally to $\mathscr{C}(\mathbb{R}^n\times \mathbb{R}^n)$.

\begin{proposition}\label{propo:TSharmonic}
Let $T=\sum_{k\geq 1}a_k\delta_{x_k}$ be locally finite with $a_k>0$. Let $S=\sum_{k,j}w_{kj}\delta_{(x_k,x_j)}$ be a finite measure with $w_{ij}\geq 0$. For $f\in\mathscr{S}_1=\mathscr{C}_c(\mathbb{R}^n)$ we have
\begin{equation*}
\Delta_{T,S}f(x) = \frac{1}{a_k}\sum_{j\geq 1}w_{jk}(f(x_j)-f(x_k)).
\end{equation*}
Hence a function $f\in\mathscr{S}_1$ is $(T,S)$-harmonic if and only if for every $k$ the mean value formula
\begin{equation*}
f(x_k) = \frac{1}{\sum_{j\geq 1}w_{jk}}\sum_{j\geq 1}w_{kj}f(x_j)
\end{equation*}
holds.
\end{proposition}
The proof is an immediate consequence of \textit{(c)} in Proposition~\ref{propo:Kirfinite}.

For our further analysis it will be convenient to introduce here some particular cases of the above discrete situation.

Let us start with the classical finite difference scheme.
\begin{proposition}[Finite differences] \label{propo:LaplacianOperatorFiniteDifference}
Let $h>0$ be given. For $\bar{k}=(k_1,\ldots,k_n)\in \mathbb{Z}^n$, set $x_{\bar{k}}=h\bar{k}\in\mathbb{R}^n$, for every $\bar{k}$. With the above notation take $T=T_h$ with $a_{\bar{k}}=h^n$ for every $\bar{k}\in\mathbb{Z}^n$ and $S=S_h$ with $w_{\bar{k}\bar{j}}=0$ if $\abs{\bar{k}-\bar{j}}>1$ or $\bar{k}=\bar{j}$, and $w_{\bar{k}\bar{j}}=h^{n-2}$ when $\abs{\bar{k}-\bar{j}}=1$. For $\Phi\in\mathscr{S}_2$, we have
\begin{equation*}
Kir_h \Phi(l,\bar{k}) = Kir_{T_h,S_h}\Phi(l,\bar{k}) = \frac{1}{h^2}\sum_{m=1}^n \left[\Phi(l\bar{k},h(\bar{k}+\bar{e_m})+ \Phi(h\bar{k},h(\bar{k}-\bar{e_m})))\right]
\end{equation*}
where $\bar{e_m}$ is the $m$-{th} vector of the canonical basis of $\mathbb{R}^n$. For $f\in\mathscr{S}_1=\mathscr{C}_c(\mathbb{R}^n)$, taking $\Phi(x,y)=f(y)-f(x)$ we obtain the corresponding Laplacian operator
\begin{equation*}
\Delta_hf(h\bar{k}) = \sum_{m=1}^n \frac{h(\bar{k}+\bar{e_m}) - 2 f(h\bar{k}) + f(h(\bar{k}-\bar{e_m}))}{h^2}.
\end{equation*}
Moreover, the harmonic functions are those for which
\begin{equation*}
f(h\bar{k}) = \frac{1}{2n}\sum_{m=1}^n\left[f(h(\bar{k}+\bar{e_m})) + f(h(\bar{k}-\bar{e_m}))\right].
\end{equation*}
\end{proposition}
\begin{proof}
Follows directly from \textit{(c)} in Proposition~\ref{propo:Kirfinite} by noticing that
\begin{equation*}
\{\bar{j}\in\mathbb{Z}^n: \abs{\bar{k}-\bar{j}}=1\}= \{\bar{k}+\bar{e_m}: m=1,\ldots,n\}\cup\{\bar{k}-\bar{e_m}: m=1\ldots,n\}.
\end{equation*}
\end{proof}

Let us point out that the normalizations of $T_h$ and $S_h$ with $h^n$, do not reflect in the Kirchhoff and Laplace operators. Nevertheless, since $h^n$ is the volume of  each cube in the cubic partition naturally induced by the sequence $\{h\bar{k}: \bar{k}\in\mathbb{Z}^n\}$, the measure $T_h$ is an approximation, in the weak convergence, of Lebesgue measure on $\mathbb{R}^n$ and has to be considered when we have a more abstract non-translation invariant setting.

The next example of the general situation is a discretization of the fractional Laplacian in $\mathbb{R}^n$.

\begin{proposition}[Discrete Fractional Laplacian]\label{propo:DiscreteFractionalLaplacian}
Let $\{x_{\bar{k}}:\bar{k}\in\mathbb{Z}^n\}$ and $T_h$ be as in Proposition~\ref{propo:LaplacianOperatorFiniteDifference}. Let $\alpha>0$ be given. Set $w_{\bar{k}\bar{j}}^\alpha=0$ if $\bar{k}=\bar{j}$ and $w_{\bar{k}\bar{j}}^\alpha=h^{n-\alpha}\frac{1}{\abs{\bar{k}-\bar{j}}^{n+\alpha}}$ if $\bar{k}\neq\bar{j}$. Let $S^\alpha_h = \sum_{\bar{k}\bar{j}}w^\alpha_{\bar{k}\bar{j}}\delta_{(h\bar{k},h\bar{j})}$. For $\Phi\in\mathscr{S}_2$ we have
\begin{equation*}
Kir_{\alpha,h}\Phi = Kir_{T_h,S^\alpha_h}\Phi = \frac{1}{h^\alpha}\sum_{\bar{j}\neq\bar{k}}\frac{\Phi(h\bar{k},h\bar{j})}{\abs{\bar{k}-\bar{j}}^{n+\alpha}}.
\end{equation*}
For $f\in\mathscr{S}_1$, with $\Phi(x,y)=f(y)-f(x)$, the above series is still convergent and
\begin{equation*}
\Delta^\alpha_h f(h\bar{k}) = \frac{1}{h^\alpha}\sum_{\bar{j}\neq\bar{k}}\frac{f(h\bar{j})-f(h\bar{k})}{\abs{\bar{k}-\bar{j}}^{n+\alpha}}.
\end{equation*}
Moreover, a function $f$ is $\alpha$ harmonic if and only if for every $\bar{k}\in\mathbb{Z}^n$ we have
\begin{equation*}
f(h\bar{k}) =\frac{1}{c(\alpha)}\sum_{\bar{j}\neq\bar{k}}\frac{f(h\bar{j})}{\abs{\bar{k}-\bar{j}}^{n+\alpha}}
\end{equation*}
with $c(\alpha)=\sum_{\bar{j}\neq\bar{0}}\abs{j}^{-n-\alpha}$.
\end{proposition}
\begin{proof}
Follows from \textit{(c)} in Proposition~\ref{propo:Kirfinite} and the absolute convergence of the involved series.
\end{proof}
The examples in Propositions~\ref{propo:LaplacianOperatorFiniteDifference} and \ref{propo:DiscreteFractionalLaplacian} are both given in terms of measures, nevertheless for $h\to 0$ we have to leave the measure space setting to allow more general distributions. We shall come back to this issue later on.

\section{Dirac deltas in metric measure spaces}\label{sec:DiracDeltasMetricMeasureSpaces}

In this section we shall briefly extend the constructions provided in Section~\ref{sec:DiracDeltasEuclideanSpaces} to metric measure spaces. It is clear that we should overcome several restrictions in the new setting. In particular we loss translation invariance and homogeneity. Nevertheless there exist in the literature regarding analysis on spaces of homogeneous type, and more general non-doubling structures, some geometric constructions which will help us in our approach to that extension. It is also worthy mention that most of the theory can be given for quasi-metrics, not just for metrics. Two reasons advise for an approach based on metric spaces. The first is simplicity. The second, more deep, is that a well known theorem due to Mac\'ias and Segovia \cite{MaSe79Lip} shows that every quasi-metric is equivalent to a power of a metric.

Let $(X,d)$ be a complete metric space. Assume that there exists an $N$ such that no $d$-ball of radius $r$ in $X$ contains more than $N$ points of any $\tfrac{r}{2}$-disperse set $D_{r/2}$ in $X$. We say that $D_\varepsilon\subset X$ is $\varepsilon$-disperse if $d(x,y)\geq \varepsilon$ for every choice of $x$ and $y$ in $D_\varepsilon$. This property of $X$ is in fact a property of finite dimension, metric or Assouad finite dimension. By taking maximal $\varepsilon$-disperse sets in $X$ we obtain $\varepsilon$-nets in $X$. And since each of them is locally finite, the space $(X,d)$ is separable. But, most important for our purposes is the existence of dyadic type nested partitions with metric control in $X$. For the case of spaces of homogeneous type, which entails the existence of a doubling Borel measure on $X$, the construction was given by M.~Christ in \cite{Christ90}. The extension to metric spaces with finite Assouad dimension is simple. The basic result is the existence of dyadic nested families satisfying all the properties of the following definition.

\begin{definition}
A dyadic family in $(X,d)$ is a countable family $\mathcal{D}=\bigcup_{j\in\mathbb{Z}}\mathcal{D}^j$ of Borel subsets of $X$ such that there exists $0<\delta<1$, constants $a<b$ and $M\in\mathbb{N}$ and a sequence $\{x^j_k: k\in K_j\}$ with $K_j$ an initial interval of positive integers which could be all $\mathbb{Z}^+$, in $X$ satisfying:
\begin{enumerate}[(D1)]
\item each $\mathcal{D}^j$ is a disjoint partition of $X$;
\item for each $Q^j_k\in\mathcal{D}^j$ we have that $B_d(x^j_k, a\delta^j)\subseteq Q^j_k\subseteq B_d(x^j_k,b\delta^j)$;
\item each $Q^j_k\in\mathcal{D}^j$ can be written as the disjoint union of at most $M$ sets $Q^{j+1}_l$ in $\mathcal{D}^{j+1}$.
\end{enumerate}
\end{definition}
Once we have a dyadic family on $X$ the nets $\{x^j_k\}$ of points in $X$ inherit at least two ways, which are essentially different, to interpret the idea of neighbor. The first is given by the metric $d$ and the second by the ancestry induced by the three structure of the dyadic system $\mathcal{D}$.

In the current general setting given by $(X,d)$, taking $\mathscr{S}_1=\mathscr{C}_c(X,d)$ the space of continuous and compactly supported functions in $X$ and $\mathscr{S}_2=\mathscr{C}_c(X\times X)$ the space of continuous and compactly supported functions in $X\times X$, and $\{x_k\}$ a sequence in $X$, the construction of $T$ and $S$ and Propositions~\ref{propo:Kirfinite} and \ref{propo:TSharmonic} of the previous section, hold \textit{mutatis mutandis}. For the sake of completeness and clearness we collect these basic results in the following statement.

\begin{proposition}\label{propo:TSLaplacianMeasureDistributions}
Let $(X,d)$ be a complete metric space. Let $\mathscr{S}_1=\mathscr{C}_c(X)$ and $\mathscr{S}_2(X\times X)$. Let $\{a_k\}$ be a sequence of positive real numbers which is locally finite with respect to $\{x_k\}$ (i.e. $\sum_{\{k: d(x_k,x_0)<R\}} a_k<\infty$ for every $x_0\in X$ and every $R>0$). Let $\{w_{kj}\}$ be a sequence of nonnegative real numbers that is locally finite with respect to the sequence $\{(x_k,x_j): k,j\}\subseteq X\times X$. Set $T = \sum_k a_k\delta_{x_k}$ and $S=\sum_{k,j}w_{kj}\delta_{(x_k,x_j)}$. Then
\begin{enumerate}[(A)]
\item for $\Phi\in\mathscr{S}_2$
\begin{equation*}
Kir_{T,S}\Phi(x_k)=\frac{1}{a_k}\sum_{j\geq 1}w_{kj}\Phi(x_k,x_j);
\end{equation*}
\item if $S$ is a finite measure and $f\in\mathscr{S}_1$,
\begin{equation*}
\Delta_{T,S}f(x_k) = \frac{1}{a_k}\sum_{j\geq 1}w_{kj} (f(x_j)-f(x_k)).
\end{equation*}
\end{enumerate}
\end{proposition}

We aim to apply the above general result to obtain analogues of Propositions~\ref{propo:LaplacianOperatorFiniteDifference} and \ref{propo:DiscreteFractionalLaplacian}, involving the geometric and measure theoretic properties of the underlying space $X$.

We shall keep working in a complete metric space with finite metric dimension equipped with a Borel measure $\mu$ which is finite and positive on the $d$-balls.

\begin{proposition}\label{propo:FiniteLaplacianGeometricMeasure}
Let $(X,d)$ be a complete metric space with finite metric dimension. Let $\mathcal{D}$ be a dyadic family in $X$. Let $\mu$ be a positive Borel measure on $X$ which is positive and finite on $d$-balls. For $j$ fixed, set $T_j=\sum_k \mu(Q^j_k)\delta_{x^j_k}$
\begin{enumerate}[(\ref{propo:FiniteLaplacianGeometricMeasure}.1)]
\item Let $S_j=\sum_{k,i}w^j_{ki}\delta_{(x^j_k,x^j_i)}$ with
    \begin{equation*}
    w^j_{ki} = H^j_{ki} \cdot
    \begin{cases}
    0 & \textrm{ if } k=i;\\
    \mu(Q^j_k)+\mu(Q^j_i) & \textrm{ if } k\neq i \textrm{ and } d(x^j_k,x^j_i) < C\delta^j;\\
    0 & \textrm{ if } d(x^j_k,x^j_i)\geq C\delta^j,
    \end{cases}
    \end{equation*}
    where $C$ is a fixed constant and for each $j$, $H^j_{ki}$ is a positive symmetric matrix. Then, for $\Phi\in\mathscr{S}_2$ we have
    \begin{equation*}
    Kir_j \Phi(x^j_k) = \frac{1}{\mu(Q^j_k)}\sum_{\{i\neq k:d(x^j_i,x^j_k)<C\delta^j\}}\Phi(x^j_k,x^j_i)
    \left(\mu(Q^j_k)+\mu(Q^j_i)\right) H^j_{ki}.
    \end{equation*}
\item For $f$ continuous and bounded on $X$, we have
\begin{align*}
\Delta_j f(x^j_k) &=  \frac{1}{\mu(Q^j_k)}\sum_{\{i:d(x^j_i,x^j_k)<C\delta^j\}}
\left(f(x^j_i) - f(x^j_k)\right)
    \left(\mu(Q^j_k)+\mu(Q^j_i)\right) H^j_{ki}\\
&= \sum_{\{i:d(x^j_i,x^j_k)<C\delta^j\}}
\left(f(x^j_i) - f(x^j_k)\right)
    \left(1+\frac{\mu(Q^j_i)}{\mu(Q^j_k)}\right) H^j_{ki}.
\end{align*}
\end{enumerate}
\end{proposition}
\begin{proof}
Follows from Proposition~\ref{propo:TSLaplacianMeasureDistributions} for the particular choice of $T_j$ and $S_j$
\end{proof}

Let us observe that the factor $H^j_{ki}$ defining $w^j_{ki}$ allows to have Proposition~\ref{propo:LaplacianOperatorFiniteDifference} as a particular case of Proposition~\ref{propo:TSLaplacianMeasureDistributions}, with $H^j_{ki}=\delta^{-2j}$ and $h=\delta^j$.

Regarding the extension to our more general geometric setting of the discrete fractional Laplacian, let us say that we have several points of view for the term $\abs{\bar{k}-\bar{j}}^{n+\alpha}$ defining the weights $w_{\bar{k}\bar{j}}$ in the Eucliden case. Among them we shall adopt the mixed one, where $\abs{h\bar{k}-h\bar{j}}^n=h^n\abs{\bar{k}-\bar{j}}^n$ is seen as the volume of the cube of side $h\abs{\bar{k}-\bar{j}}$ and $\abs{h\bar{k}-h\bar{j}}^\alpha$ as the distance between the points $h\bar{k}$ and $h\bar{j}$ of the given net.

\begin{proposition}
Let $(X,d)$, $\mathcal{D}$, $\mu$ and $T$, be as in Proposition~\ref{propo:FiniteLaplacianGeometricMeasure}. Let $\alpha>0$ and $j$ an integer be given. Set $w^{j,\alpha}_{kk}=0$ and for $i\neq k$,
\begin{equation*}
w^{j,\alpha}_{ki} = \frac{\mu(Q^j_k)\mu(Q^j_i)}{d^\alpha(x^j_k,x^j_i)\left[\mu(B(x^j_k,d(x^j_k,x^j_i)))+\mu(B(x^j_i,d(x^j_k,x^j_i)))\right]}. \end{equation*}
Let $S^\alpha_j = \sum_{k,i}w^{j,\alpha}_{ki}\delta_{(x^j_k,x^j_i)}$. Then, for $\phi\in\mathscr{S}_2$,
\begin{equation*}
Kir_{\alpha,j}\Phi(x^j_k) =\sum_{i\neq k} \frac{\mu(Q^j_i)\Phi(x^j_k,x^j_i)}{d^\alpha(x^j_k,x^j_i)\left[\mu(B(x^j_k,d(x^j_k,x^j_i)))+\mu(B(x^j_i,d(x^j_k,x^j_i)))\right]}.
\end{equation*}
For $f$ bounded we also have
\begin{equation*}
\Delta_{\alpha,j}f(x^j_k)=\sum_i \frac{\mu(Q^j_i)(f(x^j_i)-f(x^j_k))}{d^\alpha(x^j_k,x^j_i)\left[\mu(B(x^j_k,d(x^j_k,x^j_i)))+\mu(B(x^j_i,d(x^j_k,x^j_i)))\right]}.
\end{equation*}
\end{proposition}
Again the proof is just substitution in Proposition~\ref{propo:TSLaplacianMeasureDistributions} of the measures $T_j$ and $S^\alpha_j$.

Let us point out that if the space $(X,d,\mu)$ is $\gamma$-Ahlfors, which means that $\mu(B_d(x,r))\simeq r^\gamma$ with  fixed constants, then the above formula for the Laplacian takes the form
\begin{equation*}
\Delta_{\alpha,j}f(x^j_k)\simeq \sum_i \frac{\left(f(x^j_i)-f(x^j_k)\right)}{\left(d(x^j_k,x^j_i)\right)^{\gamma+\alpha}}\delta^j.
\end{equation*}
Which can be considered, for $j$ large enough, a good approximation of
\begin{equation*}
\Delta_\alpha f(x) = \int_{y\in X} \frac{f(y)-f(x)}{d(x,y)^{\gamma+\alpha}} d\mu(y),
\end{equation*}
the fractional Laplacian on $(X,d,\mu)$ when $f$ has some Lipschitz type regularity.

As we said before, when a dyadic system like $\mathcal{D}$ is given in a measure space there is still another idea of neighborhood based in ancestry instead of the distance $d$ itself. Let us briefly introduce it. Assume that $X$ is a quadrant for $\mathcal{D}$ in the sense that the union of all ancestors of any $Q^j_k\in\mathcal{D}$ is the whole space $X$. From property \textit{(D2)} of $\mathcal{D}$ we clearly have that for $x$ and $y$ in $X$ with $x\neq y$ there exists $j$ large enough such that $x\in Q^j_k$ and $y\in Q^j_i$ with $i\neq k$. So that $\rho(x,y)=\inf \{\mu(Q): Q\in\mathcal{D}, x\in Q \textrm{ and } y\in Q\}$ is well defined, positive and is actually a minimum. From the properties of the dyadic cubes in $\mathcal{D}$, it is easy to see that $\rho$ is a metric on $X$. If the space $(X,\mu)$ has no atoms then $B_\rho(x,r)=\{y:\rho(x,y)<r\}=Q$, the largest dyadic cube $Q\in\mathcal{D}$ containing $x$ such that $\mu(Q)<r$. If we consider now $X$ equipped with the new metric $\rho$ instead of $d$ and the measure $\mu$, we can define corresponding discrete Laplace type operators.

For the sake of simplicity we shall only describe this approach in $\mathbb{R}^+$ with the standard dyadic intervals and Lebesgue measure. That is $X=\mathbb{R}^+=\{x\geq 0, x\in\mathbb{R}\}$, $\mathcal{D}=\bigcup_{j\in\mathbb{Z}}\mathcal{D}^j$, $\mathcal{D}^j=\{I^j_k: [k2^{-j},(k+1)2^{-j}): k=0,1,2,\ldots\}$, $\rho(x,y)=\inf\{\abs{I}: I\in\mathcal{D} \textrm{ and } x,y\in I\}$. Notice $\abs{x-y}\leq\rho(x,y)$ but their are certainty not equivalent. Set $x^j_k=k2^{-j}$, so that $[x^j_k,x^j_{k+1})=I^j_k$ and $x^j_k\in I^j_k$ but $x^j_{k+1}$ does not. As we observed before in the general setting $B_\rho(x,r)=I\in\mathcal{D}$, where $I$ is the largest dyadic interval containing $x$ such that the length $\abs{I}$ of $I$ is less than $r$. Notice also that the $\rho$-balls have measure equal to $2^{-m}$ for some integer $m\in\mathbb{Z}$.

\begin{proposition}\label{propo:DyadicDiscreteLaplacianPoints}
Let $X=\mathbb{R}^+$, $\rho$ the dyadic distance, $\abs{E}$ the Lebesgue measure of $E$ and $x^j_k=k2^{-j}$, $k=0,1,\ldots$ Assume that $j$ is fixed. Let $T_j=\sum_{k\geq 0}2^{-j}\delta_{x^j_k}=2^{-j}\sum_{k\geq 0}\delta_{x^j_k}$. Set $w^j_{ki}=0$ if $k=i$ and if $\rho(x^j_k,x^j_i)>2^{-j+2}$ and $w^j_{ki}=2^{-j}$ when $\rho(x^j_k,x^j_i)\leq 2^{-j+2}$. Then for $f$ continuous and bounded on $\mathbb{R}^+$ we have
\begin{align*}
\Delta_{\rho,j}f(x^j_k) &= \sum_{i: \rho(x^j_k,x^j_i)\leq 2^{-j+2}}\left(f(x^j_i)-f(x^j_k)\right)\\
&= \Bigl(\sum_{i: \rho(x^j_k,x^j_i)\leq 2^{-j+2}}f(x^j_i)\Bigr)-3 f(x^j_k).
\end{align*}
\end{proposition}
Let us point out that the right hand side of the above identity can be explicitly written in terms of the indices of the sequence $x^j_i$. In fact, if $k= 4l+m$ with $m\in \{0,1,2,3\}$ and $\{m_1,m_2,m_3\}=\{0,1,2,3\}\setminus\{m\}$, we see that
\begin{equation*}
\Delta_{\rho,j}f(x^j_k) = f(x^j_{4l+m_1}) + f(x^j_{4l+m_2}) + f(x^j_{4l+m_3}) - 3 f(x^j_{4l-m}).
\end{equation*}
The discrete fractional Laplacian associated to the metric $\rho$ takes the form described in the next statement.

\begin{proposition}\label{propo:DiscreteFractionalLaplacianMetricRho}
Let $X$, $\rho$, $\{x^j_k: k\in\mathbb{Z}^+\}$ and $T_j$ be as in Proposition~\ref{propo:DyadicDiscreteLaplacianPoints}. Let $\alpha>0$ be given. Define $S^\alpha_{j\rho}$ through the sequence $w^{j,\alpha}_{ki}=0$ if $k=i$ and for $k\neq i$
\begin{equation*}
w^{j,\alpha}_{ki} = \frac{4^{-j}}{\rho(x^j_k,x^j_i)^{1+\alpha}.}
\end{equation*}
Then for $f$ bounded and continuous we have
\begin{equation*}
\Delta^\rho_{\alpha,j} f(x^j_k) = \sum_{i\geq 0}\frac{f(x^j_i)-f(x^j_k)}{\rho(x^j_k,x^j_i)} 2^{-j}.
\end{equation*}
\end{proposition}
Again, for $f\in Lip_{\beta,\rho}(\mathbb{R}^+)$ with some $\beta>\alpha$, these sequences are good approximations of the fractional Laplacian with  respect to the dyadic metric
\begin{equation*}
\Delta_{\rho,\alpha} f(x) = \int_{\mathbb{R}^+} \frac{f(y)-f(x)}{\delta(x,y)^{1+\alpha}} dy,
\end{equation*}
whose spectral theory is known (see \cite{AiBoGo13}, \cite{AcAiBoGo16}).

\section{General measures}\label{sec:GeneralMeasures}
Let $X$ be a locally compact space. Let $\mathscr{S}_1=\mathscr{C}_c(X)$ the space of compactly supported continuous functions in $X$. Let $\mathscr{S}_2=\mathscr{C}_c(X\times X)$ be the space of compactly supported continuous function defined in $X\times X$. Borel measures which are finite on compact sets provide distributions in $\mathscr{S}_1^{'}$ and $\mathscr{S}_2^{'}$. Let $\mu$ be a Borel measure on $X$ which is finite on compact sets of $X$ and $\pi$ a Borel measure on $X\times X$ which is finite on compact sets of $X\times X$. As usual set $T_\mu$ and $S_\pi$ to define the distributions $\proin{T_\mu}{\varphi}=\int_X \varphi d\mu$ for $\varphi\in\mathscr{S}_1$ and $\langle\langle S_\pi,\Phi\rangle\rangle=\iint_{X\times X}\Phi d\pi$ for $\Phi\in\mathscr{S}_2$. In this case the formula defining the Kirchhoff divergence of $\Phi$ in \textit{(1.2.b)} becomes,
\begin{equation}\label{eq:formulaKirchhoffGeneralMeasures}
\int_X \varphi\psi d\mu = \iint_{X\times X} \varphi \Phi d\pi
\end{equation}
for every $\varphi\in\mathscr{C}_c(X)$.

It is worthy noticing that \eqref{eq:formulaKirchhoffGeneralMeasures} may have no solution. In fact, let $X=[0,1]$ with its standard metric structure. Take $\mu=\delta_0$ and $d\pi= dx dy$ in the unit square. If $\Phi\equiv 1$ the right hand side of \eqref{eq:formulaKirchhoffGeneralMeasures} is $\iint_{[0,1]^2}\varphi(x) dx dy = \int_{[0,1]}\varphi(x) dx$. The left hand side of \eqref{eq:formulaKirchhoffGeneralMeasures} for $\psi\in\mathscr{C}[0,1]$, instead $\int_{[0,1]}\varphi\psi d\delta_0=\varphi(0)\psi(0)$. Taking $\varphi$ with $\varphi(0)=0$ and $\int_{[0,1]}\varphi dx >0$ we see that \eqref{eq:formulaKirchhoffGeneralMeasures} can not hold. On the other hand, it is also easy to observe that non uniqueness of solution of \eqref{eq:formulaKirchhoffGeneralMeasures} is possible. Let $X=[0,1]$, $\mu=\delta_0$, $\pi=\delta_0\times\delta_0$ and $\Phi\equiv 1$. Since the left hand side of \eqref{eq:formulaKirchhoffGeneralMeasures} is again given by $\varphi(0)\psi(0)$ and the right hand side is now given by
\begin{equation*}
\iint_{[0,1]^2}\varphi\Phi \,d\pi = \iint_{[0,1]^2}\varphi\, d(\delta_0\times\delta_0) = \varphi(0).
\end{equation*}
So that any continuous $\psi$ with $\psi(0)=1$ solves \eqref{eq:formulaKirchhoffGeneralMeasures}.
Some particular cases of existence and uniqueness for coupling probability measures are given in \cite{AiGoKirchhoff1}.

Two somehow extremal situations of the relation between $\mu$ and $T$ are provided by the probabilistic concepts of independence and determinism. The next two results point in each one of these directions.
\begin{proposition}\label{propo:IndependenceCaseKirchhoffLaplacian}
Let $X$, $\mathscr{S}_1$, $\mathscr{S}_2$ and $\mu$ as before. Assume that $\pi=\pi_1\times\pi_2$ with $\pi_1<<\mu$ and $\pi_2$ a positive measure which is finite on the compacts of $X$. Then, for $\Phi\in\mathscr{S}_2$ we have
\begin{equation*}
Kir_{\mu,\pi}\Phi(x) = \frac{d\pi_1}{d\mu}(x) \int_{y\in X}\Phi(x,y) d\pi_2(y),
\end{equation*}
where $\frac{d\pi_1}{d\mu}$ denotes the Radon-Nikodym derivative of $\pi_1$ with respect to $\mu$. Moreover, if $\pi_2(X)<\infty$, for $f$ continuous and bounded we have
\begin{equation*}
\Delta_{\mu,\pi} f(x) = \frac{d\pi_1}{d\mu}(x)\left(\int_{y\in X} f(y) d\pi_1(x) - f(x)\pi_2(X)\right).
\end{equation*}
\end{proposition}

\begin{proof}
\begin{align*}
\int_{y\in X}\varphi(x)\psi(x) d\mu(x) &=
\int_{x\in X}\varphi(x)\left(\int_{y\in X} \Phi(x,y) d\pi_2(y)\right) d\pi_1(x)\\
&= \int_{x\in X}\varphi(x)\frac{d\pi_1}{d\mu}(x)\left(\int_{y\in X} \Phi(x,y) d\pi_2(y)\right) d\mu(x)
\end{align*}
for every $\varphi\in\mathscr{S}_1$.
\end{proof}
The next result concerns the deterministic case.

\begin{proposition}\label{propo:DeterministicCase}
Let $X$ be a locally compact space and let $\mu$ be a Borel measure on $X$ which is finite on compact sets. Let $F:X\to X$ be a continuous transformation of $X$. Set $G: X\to X\times X$, $G(x)=(x,F(x))$. Let $\pi$ be the measure defined on the Borel subsets of $X\times X$ by $\pi=\mu\circ G^{-1}$, in other words
\begin{equation*}
\pi(E) = \mu(G^{-1}(E))
\end{equation*}
for $E$ any Borel set in $X\times X$. Then for $\Phi\in\mathscr{C}_c(X\times X)$ we have
\begin{equation*}
Kir_{\mu,F}\Phi(x) = \Phi(x,F(x))
\end{equation*}
and for $f$ continuous on $X$,
\begin{equation*}
\Delta_{\mu,F} f = f\circ F - f.
\end{equation*}
\end{proposition}
\begin{proof}
Notice first that, for $\Theta\in\mathscr{C}_c(X\times X)$ we have
\begin{equation*}
\iint_{X\times X} \Theta d\pi = \int_X \Theta (x,F(x)) d\mu(x).
\end{equation*}
This follows from the standard arguments noticing that for $\Theta=\mathcal{X}_E$, $E$ a Borel subset of $X\times X$, the formula is nothing but the definition of $\pi$. Hence the right hand side of \eqref{eq:formulaKirchhoffGeneralMeasures} can be written as
\begin{equation*}
\iint_{X\times X}\varphi(x)\Phi(x,y) d\pi(x,y) = \int_X \varphi(x) \Phi(x,F(x)) d\mu(x)
\end{equation*}
for every $\varphi\in\mathscr{C}_c(X)$. Hence
\begin{equation*}
Kir_{\mu,F}\Phi = \Phi\circ G
\end{equation*}
as desired.
\end{proof}

For further reference, notice that if for $h>0$ we take in Proposition~\ref{propo:DeterministicCase}, $\pi_h=\tfrac{1}{h} \mu\circ G^{-1}$, we would have the Laplace operator given by
\begin{equation*}
\Delta_{\mu,F,h}f = \frac{1}{h} (f\circ F - f).
\end{equation*}

\section{$S$ of positive order. The case of the derivatives of deterministic couplings}\label{sec:SpositiveOrderDeterministicCoupling}

So far we have only considered measures, i.e. distributions of order zero even when our general point of view in the introduction is given in terms of general distributions of Schwartz type.

Let $\mathscr{S}_1=\mathscr{C}_c^\infty(\mathbb{R}^n)$, $\mathscr{S}_2=\mathscr{C}_c^\infty(\mathbb{R}^n\times \mathbb{R}^n)$. Assume a $\mathscr{C}^1$ mapping $F:\mathbb{R}^n\to\mathbb{R}^n$ is given. Let $\mu$ be a Borel measure on $\mathbb{R}^n$ that is finite on compact sets. Set, as in Proposition~\ref{propo:DeterministicCase}, $\pi=\mu\circ G^{-1}$ with $G(x)=(x,F(x))$, $x\in\mathbb{R}^n$. Taking in \textit{(1.2.b)} $T=T_\mu$ as before and $S$ to be some distributional partial derivative of $\pi$, we may find explicit formulas for Krichhoffean and Laplacean operators. Let us state precisely  the setting and the results.
\begin{proposition}
Let $F$, $\mu$ and $T$ as before. Assume that $d\mu(x)= g(x) dx$ with $g$ smooth and positive. Let $\proin{T_\mu}{\varphi}=\int_{\mathbb{R}^n}\varphi(x) g(x) dx$. Set $S^1_i=\frac{\partial \pi}{\partial x_i}$ and $S^2_i=\frac{\partial\pi}{\partial y_i}$, $i=1,\ldots,n$, where the partial derivatives are considered in the sense of distributions in $\mathbb{R}^{2n}$. Then, for $\Phi\in\mathscr{S}_2$,
\begin{eqnarray*}
Kir_{i,1}\Phi &=& \frac{1}{g}\frac{\partial}{\partial x_i}[g\cdot(\Phi\circ G)]- \frac{\partial \Phi}{\partial x_i}\circ G \\
Kir_{j,2}\Phi &=& \frac{\partial\Phi}{\partial y_j}\circ G.
\end{eqnarray*}
For $f$ smooth
\begin{eqnarray*}
\Delta_{i,1} f &=& \frac{\partial F}{\partial x_i}\cdot\left(\nabla f \circ F\right) + \bigl[(f\circ F) - f\bigr]\frac{\partial}{\partial x_i}\left(\log g\right);\\
\Delta_{j,2} f &=& \frac{\partial f}{\partial x_j}\circ F.
\end{eqnarray*}
\end{proposition}
\begin{proof}
We have to check \textit{(1.2.b)} which in the current situation reads
\begin{equation*}
\int_{\mathbb{R}^n}\varphi(x)\psi(x) g(x) dx = -\iint_{\mathbb{R}^{2n}}\frac{\partial }{\partial x_i}(\varphi\Phi)(x,y) d\pi(x,y).
\end{equation*}
Let us write out the right hand side above in terms of $\pi$. We have that
\begin{align*}
-\iint_{\mathbb{R}^{2n}}&\frac{\partial}{\partial x_i}(\varphi\Phi)(x,y) d\pi(x,y)\\
&= -\int_{\mathbb{R}^{n}}\frac{\partial}{\partial x_i}(\varphi\Phi)(x,F(x))g(x) d(x)\\
&= -\int_{\mathbb{R}^{n}}\frac{\partial\varphi}{\partial x_i}(x)\Phi(x,F(x))g(x) d(x)
- \int_{\mathbb{R}^{n}}\varphi(x)\frac{\partial \Phi}{\partial x_i}(x,F(x))g(x) d(x)\\
&= \int_{\mathbb{R}^{n}}\frac{\partial}{\partial x_i}[g\cdot(\Phi\circ G)](x)\varphi(x) d(x)
- \int_{\mathbb{R}^{n}}g(x)\frac{\partial \Phi}{\partial x_i}(x,F(x))\varphi(x) d(x)
\end{align*}
for every $\varphi\in\mathscr{S}_1$. Hence
\begin{equation*}
\psi = \frac{1}{g}\frac{\partial}{\partial x_i} (g\cdot\Phi\circ G)-\frac{\partial \Phi}{\partial x_i}\circ G
\end{equation*}
as desired. For the second formula take $\Phi(x,y)=f(y)-f(x)$ in the above. Then
\begin{align*}
\Delta_{i,1}f(x) &= \frac{1}{g(x)}\frac{\partial}{\partial x_i}\left[g(x)\left(f(F(x))-f(x)\right)\right]
- \frac{\partial (f(y)-f(x))}{\partial x_i}(x,F(x))\\
&= \nabla f (F(x))\cdot\frac{\partial F}{\partial x_i}(x) - \frac{\partial f}{\partial x_i}(x) +
\frac{\partial \log g(x)}{\partial x_i}(f(F(x))-f(x)) +\frac{\partial f}{\partial x_i}(x),
\end{align*}
which is the desired formula. For the derivatives with respect to the $y$ variables the calculations are even easier. For $\varphi\in\mathscr{S}_1$,
\begin{align*}
\int_{\mathbb{R}^n}\varphi\Phi g dx &= \iint_{\mathbb{R}^{2n}}\frac{\partial}{\partial y_j}(\varphi\Phi) d\pi\\
&= -\iint_{\mathbb{R}^{2n}}\varphi(x)\frac{\partial\Phi}{\partial y_j}(x,y) d\pi(x,y)\\
&= -\int_{\mathbb{R}^{n}}\varphi(x)\frac{\partial\Phi}{\partial y_j}(x,F(x))g(x) dx,
\end{align*}
which proves the desired formula for $Kir_{j,2}\Phi$.
\end{proof}

\section{The classical Laplacian in $\mathbb{R}^n$ and formula \textit{(1.2.b)}}\label{sec:ClassicalLaplacianRn}

In this brief section we search for a couple of distributions $T$ and $S$ in $\mathbb{R}^n$ and $\mathbb{R}^{2n}$ respectively, such that the classical Laplacian $\Delta f= \sum_{i=1}^n \frac{\partial^2 f}{\partial x_i^2}$, can be seen as $Kir_{T,S}(f(y)-f(x))$, in the sense of \textit{(1.2.b)}. It is important at this point to emphasize that we are not trying to define the Laplacian. Instead, given the classical Laplacian, we are providing two distributions $T$ and $S$ such that $Kir_{T,S}(f(y)-f(x))=\Delta f$.

\begin{proposition}
Let $\mathscr{S}_1=\mathscr{C}_c^\infty(\mathbb{R}^n)$ and $\mathscr{S}_2=\mathscr{C}_c^\infty(\mathbb{R}^{2n})$, the classical test function spaces in $\mathbb{R}^n$ and $\mathbb{R}^{2n}$. Let $T$ be the distribution in $\mathbb{R}^n$ generated by the function identically equal to one. In other words, $\proin{T}{\varphi}=\int_{\mathbb{R}^n}\varphi(x) dx$, $\varphi\in\mathscr{S}_1$. Let $\pi$ be the measure defined in $\mathbb{R}^{2n}$ by $\pi(E)=\abs{\{x\in\mathbb{R}^n: (x,x)\in E\}}_n$, where $\abs{\cdot}_n$ denotes the $n$-dimensional Lebesgue measure. Since $\pi$ is a locally finite positive measure in $\mathbb{R}^{2n}$, it defines a distribution in $\mathbb{R}^{2n}$. Set $S=\Delta_y\pi=\sum_{j=1}^n\frac{\partial^2\pi}{\partial y_j^2}$, where the derivatives are considered in the sense of distributions. Then
\begin{equation*}
\Delta f(x) = Kir_{T,S}(f(y)-f(x)).
\end{equation*}
\end{proposition}

\begin{proof}
Let us first find $Kir_{T,S}\Phi$ for $\Phi\in\mathscr{S}_2$. To check \textit{(1.2.b)} in this case, let us start from the right hand side with $\varphi\in\mathscr{S}_1$ and $\Phi\in\mathscr{S}_2$,
\begin{align*}
\langle\langle S,\varphi\Phi\rangle\rangle &= \langle\langle \Delta_y\pi,\varphi\Phi\rangle\rangle\\
&= (-1)^2 \langle\langle \pi,\Delta_y(\varphi\Phi)\rangle\rangle\\
&= \langle\langle \pi,\varphi\Delta_y\Phi\rangle\rangle\\
&= \iint_{\mathbb{R}^{2n}}\varphi(x)(\Delta_y\Phi)(x,y) d\pi(x,y)\\
&= \int_{\mathbb{R}^{n}}\varphi(x)(\Delta_y\Phi)(x,x) dx.
\end{align*}
Since the left hand side of \textit{(1.2.b)} reads
\begin{equation*}
\proin{T}{\varphi\psi} = \int_{\mathbb{R}^n}\varphi(x) Kir_{T,S}\Phi(x) dx
\end{equation*}
and the equation has to be true for every $\varphi\in\mathscr{S}_1$, we get
\begin{equation*}
Kir_{T,S}\Phi(x) = (\Delta_y\Phi)(x,x).
\end{equation*}
For $f\in\mathscr{S}_1$ we get
\begin{equation*}
\Delta_{T,S} f(x) = Kir_{T,S}(f(y)-f(x)) = \Delta f(x),
\end{equation*}
as desired.
\end{proof}

\section{On fractional Kirchhoff Divergences in the Euclidean space}\label{sec:FractionalKirchhoffDivergencesEuclideanSpace}
The observation regarding the Kirchhoff divergence and the Laplacian in the previous section, shows that the differential character of the operator induced by $S$, assuming $T=1$, is related to the singularity of $S$ on the diagonal of $\mathbb{R}^n\times\mathbb{R}^n$. In this section we explore this fact, searching for the distributions $S$ that produce, through \textit{(1.2.b)}, Kirchhoff fractional type operators. The distributions $S_s$ in $\mathbb{R}^{2n}$ are defined in terms of the ``affinity'' $\abs{x-y}^{-(n+2s)}$ for $0<s<1$. The singularity of the kernel on the diagonal of $\mathbb{R}^n\times\mathbb{R}^n$ increases as $s$ tends to one. Actually $s=\tfrac{1}{2}$ divides the character of the singularity and hence the actual definition of the distribution $S_s$. We shall consider these two cases separately.

\subsection*{First case: $0<s<\tfrac{1}{2}$.}
Let us start by proving the convergence of the integral defining the distribution $S$.
\begin{lemma}\label{lem:VergenceFractDistributionCase1}
For $\Phi\in\mathscr{C}_c^\infty(\mathbb{R}^{2n})$ the function defined in $\mathbb{R}^{2n}$ by $\frac{\Phi(x,y)-\Phi(x,x)}{\abs{x-y}^{n+2s}}$ is in $L^1(\mathbb{R}^{2n})$. The linear functional $S:\mathscr{C}_c^\infty(\mathbb{R}^{2n})\to \mathbb{R}$ given by
\begin{equation*}
\langle\!\langle S,\Phi \rangle\!\rangle = \iint_{\mathbb{R}^{2n}}\frac{\Phi(x,y)-\Phi(x,x)}{\abs{x-y}^{n+2s}} dx dy
\end{equation*}
defines a distribution in $\mathscr{D}^{'}(\mathbb{R}^{2n})$.
\end{lemma}
\begin{proof}
Let $K$ be a compact set in $\mathbb{R}^{n}$ such that $\supp \Phi\subset K\times \mathbb{R}^{n}$. Then, with $\omega_{n-1}$ the surface area of the unit sphere of $\mathbb{R}^{n}$,
\begin{align*}
\iint_{\mathbb{R}^{2n}}&\frac{\abs{\Phi(x,y)-\Phi(x,x)}}{\abs{x-y}^{n+2s}} dx dy\\
& = \int_K\int_{\mathbb{R}^{n}} \frac{\abs{\Phi(x,y)-\Phi(x,x)}}{\abs{x-y}^{n+2s}} dx dy \\
&\leq \int_K\left\{\int_{\abs{x-y}<1} \frac{\abs{\Phi(x,y)-\Phi(x,x)}}{\abs{x-y}^{n+2s}} dy + 2\norm{\Phi}_\infty\int_{\abs{x-y}\geq 1}\frac{dy}{\abs{x-y}^{n+2s}} \right\} dx\\
&\leq \int_K\left\{\int_{\abs{x-y}<1} \frac{\norm{\nabla_y\Phi}_\infty \abs{x-y}}{\abs{x-y}^{n+2s}} dy + \frac{\omega_{n-1}\norm{\Phi}_\infty}{s}\right\} dx\\
&=\omega_{n-1}\abs{K}\left(\frac{\norm{\nabla_y\Phi}_\infty}{1-2s} + \frac{\norm{\Phi}_\infty}{s}\right).
\end{align*}
In order to prove the continuity with the topology of $\mathscr{C}_c^\infty(\mathbb{R}^{2n})$ of the linear functional
$\langle\!\langle S,\Phi \rangle\!\rangle = \iint_{\mathbb{R}^{2n}}\frac{\Phi(x,y)-\Phi(x,x)}{\abs{x-y}^{n+2s}} dx dy
$, take a sequence $\Phi_k$ that tends to zero in $\mathscr{C}_c^\infty(\mathbb{R}^{2n})$. This means that there exists a compact set $\mathbb{K}$ in $\mathbb{R}^{2n}$ containing the supports of all the $\Phi_k'$s, and $\Phi_k$ and all its derivatives converge uniformly to zero in $\mathbb{R}^{2n}$. Let now $K$ be the projection of $\mathbb{K}$ in the first variables $x=(x_1,\ldots,x_n)$. Hence
\begin{align*}
\abs{\langle\!\langle S,\Phi_k\rangle\!\rangle} & \leq \int_{K}\int_{\mathbb{R}^{n}}\frac{\abs{\Phi_k(x,y)-\Phi_k(x,x)}}{\abs{x-y}^{n+2s}} dx dy \\
&\leq \omega_{n-1}\abs{K}\left(\frac{\norm{\nabla_y\Phi_k}_\infty}{1-2s} + \frac{\norm{\Phi_k}_\infty}{s}\right)
\end{align*}
which tends to zero when $k\to \infty$.
\end{proof}

\begin{lemma}\label{lem:VergenceFractBounded}
	For $\Phi\in \mathscr{C}_c^\infty(\mathbb{R}^{2n})$ the function
	\begin{equation*}
	 \psi(x)=\int_{\mathbb{R}^n}\frac{\Phi(x,y)-\Phi(x,x)}{\abs{x-y}^{n+2s}} dy
	\end{equation*}
	belongs to $\mathscr{C}_c^\infty(\mathbb{R}^n)$ and is bounded by $\omega_{n-1}\left(\frac{\norm{\nabla_y\Phi}_\infty}{1-2s} + \frac{\norm{\Phi}_\infty}{s}\right)$.
\end{lemma}
\begin{proof}
It is clear that if $x$ does not belong to the projection, in the variable $x$, of the support of $\Phi$, we have that $\psi(x)=0$. On the other hand, as we showed in the proof of Lemma~\ref{lem:VergenceFractDistributionCase1},
$\abs{\psi(x)}\leq\omega_{n-1}\left(\frac{\norm{\nabla_y\Phi}_\infty}{1-2s} + \frac{\norm{\Phi}_\infty}{s}\right)$.
The regularity of $\psi$ follows from the fact that $\psi(x)=\int_{\mathbb{R}^n}\frac{\Phi(x,x-y)-\Phi(x,0)}{\abs{y}^{n+2s}}dx$.
\end{proof}

\begin{proposition}\label{propo:KirDivergenceFracCase1}
With $S$ as in Lemma~\ref{lem:VergenceFractDistributionCase1} in $\mathscr{D}^{'}(\mathbb{R}^{2n})$	 and $T$ the distribution induced in $\mathbb{R}^n$ by the function identically equal to one, we have
\begin{equation}\label{eq:DivFractRn}
Kir_{T,S}\Phi(x) = \psi(x) = \int_{\mathbb{R}^n}\frac{\Phi(x,y)-\Phi(x,x)}{\abs{x-y}^{n+2s}} dy
\end{equation}
for every $\Phi\in \mathscr{C}_c^\infty(\mathbb{R}^{2n})$.
\end{proposition}
\begin{proof}
In the current situation equation \textit{(1.2.b)} takes the form
\begin{align*}
\int_{\mathbb{R}^n}\varphi(x) Kir_{T,S}\Phi(x) dx &= \langle T,\varphi Kir_{T,S}\Phi\rangle\\
&= \langle\!\langle S,\varphi\Phi\rangle\!\rangle\\
&= \iint_{\mathbb{R}^{2n}}\frac{\varphi(x)\Phi(x,y)-\varphi(x)\Phi(x,x)}{\abs{x-y}^{n+2s}} dx dy\\
&= \int_{\mathbb{R}^{n}}\varphi(x)\left(\int_{\mathbb{R}^{n}}\frac{\Phi(x,y)-\Phi(x,x)}{\abs{x-y}^{n+2s}} dy\right) dx
\end{align*}
for every $\varphi\in \mathscr{C}_c^\infty(\mathbb{R}^{n})$.
\end{proof}

Notice that the boundedness of $\psi$ in Lemma~\ref{lem:VergenceFractBounded} only requires the boundedness of $\Phi$ and its gradient in the second variable $y=(y_1,\ldots,y_n)$. So that if $f$ and its gradient are bounded, we can take $\Phi(x,y)=f(y)-f(x)$ in formula \eqref{eq:DivFractRn} in order to obtain the fractional Laplacian $(-\Delta)^s$, $0<s<\tfrac{1}{2}$, as a $Kir\, grad$ operator with $grad \,f(x,y) = f(y)-f(x)$ and $Kir$ given by \eqref{eq:DivFractRn}. Precisely,
\begin{equation*}
(-\Delta)^s f(x) = \int_{\mathbb{R}^n}\frac{f(y)-f(x)}{\abs{x-y}^{n+2s}} dy.
\end{equation*}

\subsection*{Second case: $\tfrac{1}{2}\leq s <1$}  In this case the integral defining $S$ has to be taken in the principal value sense, because for $s\geq\tfrac{1}{2}$ the function $\frac{\Phi(x,y)-\Phi(x,x)}{\abs{x-y}^{n+2s}}$ is generally not integrable in $\mathbb{R}^{2n}$.

\begin{lemma}\label{lemma:VergenceFracDistributionCase2}
For every $\Phi\in\mathscr{C}_c^\infty(\mathbb{R}^{2n})$ the limit $\lim_{\varepsilon\to 0}\iint_{B_\varepsilon^c}\frac{\Phi(x,y)-\Phi(x,x)}{\abs{x-y}^{n+2s}} dx dy$ exists, where $B_\varepsilon^c$ is the complement in $\mathbb{R}^{2n}$ of the diagonal $\varepsilon$-band $B_\varepsilon=\{(x,y)\in\mathbb{R}^n\times\mathbb{R}^n: \abs{x-y}<\varepsilon\}$. Moreover, the limit defines a distribution in $\mathscr{D}^{'}(\mathbb{R}^{2n})$.
\end{lemma}
\begin{proof}
Let $\varepsilon >0$ fixed. Let us denote by $\nabla_y$ the gradient of functions defined on $\mathbb{R}^n\times\mathbb{R}^n$ with respect to $y$, the second group of variables. Notice first that $\iint_{B_\varepsilon^c}\frac{\abs{\Phi(x,y)-\Phi(x,x)}}{\abs{x-y}^{n+2s}} dx dy$ is finite. In fact, with $K$ such that $\supp\Phi \subset K\times\mathbb{R}^n$, $K$ compact,
\begin{align*}
\iint_{B_\varepsilon^c}\frac{\abs{\Phi(x,y)-\Phi(x,x)}}{\abs{x-y}^{n+2s}} dx dy & = \int_{x\in K}\left(\int_{\abs{x-y}\geq \varepsilon}\frac{\abs{\Phi(x,y)-\Phi(x,x)}}{\abs{x-y}^{n+2s}}  dy\right)dx\\
& \leq c_{n,s}\norm{\Phi}_{\infty}\abs{K}\varepsilon^{-2s}.
\end{align*}
Since, from symmetry, we have $\int_{\varepsilon\leq\abs{x-y}<1}\frac{\nabla_y\Phi(x,x)\cdot (x-y)}{\abs{x-y}^{n+2s}}dy=0$, we write
\begin{align*}
\iint_{B_\varepsilon^c}& \frac{\Phi(x,y)-\Phi(x,x)}{\abs{x-y}^{n+2s}} dx dy\\
&=\int_{x\in K}\left(\int_{\varepsilon\leq\abs{x-y}<1}\frac{\Phi(x,y)-\Phi(x,x)-\nabla_y\Phi(x,x)\cdot (y-x)}{\abs{x-y}^{n+2s}}dy\right.\\
& \phantom{\int_{x\in\Pi_1(\supp\Phi)}\left(\right.}\left.+\int_{\abs{x-y}\geq 1}\frac{\Phi(x,y)-\Phi(x,x)}{\abs{x-y}^{n+2s}}dy\right)dx.
\end{align*}
So that, for $0<\delta<\varepsilon<1$ we get
\begin{align*}
&\abs{\iint_{B_\delta^c}-\iint_{B_\epsilon^c}}\\
& =
\abs{\int_{x\in K}\left(\int_{\varepsilon>\abs{x-y}\geq \delta}\frac{\Phi(x,y)-\Phi(x,x)-\nabla_y\Phi(x,x)\cdot (y-x)}{\abs{x-y}^{n+2s}}dy\right)dx}\\
& \leq \int_{x\in K}\int_{\varepsilon>\abs{x-y}\geq \delta}\sup_{\abs{\alpha}=2}\norm{\partial^\alpha_y\Phi}_\infty
\frac{\abs{x-y}^2}{\abs{x-y}^{n+2s}}dydx\\
& = c_{n,s}\sup_{\abs{\alpha}=2}\norm{\partial^\alpha_y\Phi}_\infty (\varepsilon^{2(1-s)}-\delta^{2(1-s)}),
\end{align*}
which tends to zero for $\varepsilon\to 0$. Here $\abs{\alpha}=\sum_{i=1}^n\alpha_i$ is the length of the multiindex $\alpha$.

Let us prove that $\langle\!\langle S,\Phi \rangle\!\rangle = \lim_{\varepsilon\to 0}\iint_{B_\varepsilon^c}\frac{\Phi(x,y)-\Phi(x,x)}{\abs{x-y}^{n+2s}} dx dy$ defines a distribution in $\mathscr{D}^{'}(\mathbb{R}^{2n})$. Let $\{\Phi_k: k\in\mathbb{N}\}$ be a sequence in $\mathscr{C}_c^\infty(\mathbb{R}^{2n})$ such that $\Phi_k\to 0$ in $\mathscr{C}_c^\infty(\mathbb{R}^{2n})$. Let $\mathbb{K}$ be a compact set $\mathbb{R}^{2n}$ such that $\supp\Phi_k\subseteq \mathbb{K}$, for every $k\in\mathbb{N}$. Moreover, $\partial^\alpha\Phi_k\rightrightarrows 0$, uniformly for every $\alpha\in\mathbb{N}^{2n}_0$. With $K$ compact in $\mathbb{R}^n$ such that $K\times\mathbb{R}^n \supseteq \mathbb{K}$, we have
\begin{align*}
&\abs{\langle\!\langle S,\Phi_k\rangle\!\rangle}\\
& = \lim_{\varepsilon\to 0}\left|\int_{x\in K}\left(\int_{\varepsilon\leq\abs{x-y}<1}\frac{\Phi_k(x,y)-\Phi_k(x,x)-\nabla_y\Phi_k(x,x)\cdot (y-x)}{\abs{x-y}^{n+2s}}dy\right.\right.\\
&\phantom{\lim_{\varepsilon\to 0}\left|\int_{x\in K}\left(\right.\right.}\left.\left.+\int_{\abs{x-y}\geq 1}\frac{\Phi_k(x,y)-\Phi_k(x,x)}{\abs{x-y}^{n+2s}}dy\right)dx\right| \\
&\leq \limsup_{\varepsilon\to 0}\int_{x\in K}\int_{\varepsilon\leq\abs{x-y}<1}\frac{\abs{\Phi_k(x,y)-\Phi_k(x,x)-\nabla_y\Phi_k(x,x)\cdot (y-x)}}{\abs{x-y}^{n+2s}}dy dx \\
&\phantom{\lim_{\varepsilon\to 0}\left|\int_{x\in K}\left(\right.\right.} + c_{n,s}\norm{\Phi_k}_\infty\abs{K}\\
&\leq c_{n,s}\abs{K}\left(\sup_{\abs{\alpha}=2}\norm{\partial^\alpha\Phi_k}_\infty+\norm{\Phi_k}_\infty\right),
\end{align*}
which tends to zero as $k\to\infty.$
\end{proof}

\begin{lemma}\label{lem:LimitFractionalEpsilonWellDefined}
For $\Phi\in\mathscr{C}_c^\infty(\mathbb{R}^{2n})$, the function 
\begin{equation*}
\psi(x)=\lim_{\varepsilon\to 0}\int_{\abs{x-y}\geq \varepsilon}\frac{\Phi(x,y)-\Phi(x,x)}{\abs{x-y}^{n+2s}} dy
\end{equation*}
is well defined as a continuous and compactly supported function on $\mathbb{R}^n$.
\end{lemma}
\begin{proof}
Set $\psi_\varepsilon(x)=\int_{\abs{x-y}\geq \varepsilon}\frac{\Phi(x,y)-\Phi(x,x)}{\abs{x-y}^{n+2s}} dy$. Let us prove that $\psi_\varepsilon$ is a Cauchy sequence in the uniform norm on the compact $K$ if $K\times\mathbb{R}^n\supseteq\supp\Phi$. In fact, since
\begin{align*}
\psi_\varepsilon(x) & =
\int_{\abs{x-y}\geq \varepsilon}\frac{\Phi(x,y)-\Phi(x,x)}{\abs{x-y}^{n+2s}} dy\\ 
& = \int_{1>\abs{x-y}\geq \varepsilon} + \int_{\abs{x-y}\geq 1}\\
& = \int_{1>\abs{x-y}\geq\varepsilon}\frac{\Phi(x,y)-\Phi(x,x)-\nabla_y\Phi(x,x)\cdot (y-x)}{\abs{x-y}^{n+2s}}dy\\
& \phantom{\int_{1>\abs{x-y}\geq\varepsilon}}
+\int_{\abs{x-y}\geq 1}\frac{\Phi(x,y)-\Phi(x,x)}{\abs{x-y}^{n+2s}}dy,
\end{align*}
for $0<\delta<\varepsilon<1$, se have
\begin{align*}
\abs{\psi_\delta(x) - \psi_\varepsilon (x)} &=
\abs{\int_{\abs{x-y}\geq \delta}-\int_{\abs{x-y}\geq \varepsilon}\frac{\Phi(x,y)-\Phi(x,x)}{\abs{x-y}^{n+2s}}dy}\\
&\leq \int_{\delta\leq\abs{x-y}<\varepsilon}\frac{\abs{\Phi(x,y)-\Phi(x,x)-\nabla_y\Phi(x,x)\cdot (y-x)}}{\abs{x-y}^{n+2s}} dy\\
&\leq \int_{\delta\leq\abs{x-y}<\varepsilon}\sup_{\abs{\alpha}=2}\norm{\partial^\alpha_y\Phi}_\infty\frac{\abs{x-y}^2}{\abs{x-y}^{n+2s}} dy\\
&\leq c_{n,s}\sup_{\abs{\alpha}=2}\norm{\partial^\alpha_y\Phi}_\infty (\varepsilon^{2(1-s)}-\delta^{2(1-s)})
\end{align*}
for every $x\in K$. Hence $\psi_\varepsilon$ it converges to a continuous function supported in $K$ as $\varepsilon\to 0$.
\end{proof}

\begin{proposition}\label{propo:KirFracRnIdenticallyOne}
For $T$, the distribution in $\mathbb{R}^n$ induced by the function identically equal to one, and $S$ as in Lemma~\ref{lemma:VergenceFracDistributionCase2}, we have that for $\Phi\in\mathscr{C}_c^\infty(\mathbb{R}^{2n})$
\begin{equation*}
Kir\,\Phi (x) = \lim_{\varepsilon\to 0^+}\int_{\abs{x-y}\geq \varepsilon}\frac{\Phi(x,y)-\Phi(x,x)}{\abs{x-y}^{n+2s}}dy
\end{equation*}
solves equation \textit{(1.2.b)}.
\end{proposition}
\begin{proof}
With $\varphi\in\mathscr{C}_c^\infty(\mathbb{R}^{n})$ and $\Phi\in\mathscr{C}_c^\infty(\mathbb{R}^{2n})$, equation \textit{(1.2.b)} takes the form
\begin{align*}
\int_{\mathbb{R}^n}\varphi(x)Kir\,\Phi(x)dx &= \langle\!\langle S,\varphi\Phi\rangle\!\rangle \\
& = \lim_{\varepsilon\to 0^+}\iint_{B^c_\varepsilon}\frac{\varphi(x)\Phi(x,y)-\varphi(x)\Phi(x,x)}{\abs{x-y}^{n+2s}}dxdy\\
&=\lim_{\varepsilon\to 0^+}\int_{\mathbb{R}^n}\varphi(x)\left(\int_{\abs{x-y}\geq\varepsilon}\frac{\Phi(x,y)-\Phi(x,x)}{\abs{x-y}^{n+2s}}dy\right)dx\\
&=\int_{\mathbb{R}^n}\varphi(x)\left(\lim_{\varepsilon\to 0^+}\int_{\abs{x-y}\geq\varepsilon}\frac{\Phi(x,y)-\Phi(x,x)}{\abs{x-y}^{n+2s}}dy\right)dx
\end{align*}
the last equation follows from Lebesgue dominated convergence theorem and Lemma~\ref{lem:LimitFractionalEpsilonWellDefined}.
\end{proof}

\section{Fractional Kirchhoff divergence on Ahlfors spaces}\label{sec:FractionalKirchhoffdivergenceAhlforsSpaces}

The case $0<s<\tfrac{1}{2}$ in the previous section admits an extension to Ahlfors regular metric spaces. Let us fix the basic notation. Let $(X,d)$ be a metric space and $\mu$ be a Borel measure on $X$ such that there exist constants $0<c_1\leq c_2<\infty$, $\gamma>0$, for which the inequalities
\begin{equation*}
c_1 r^\gamma\leq \mu(B(x,r))\leq c_2 r^\gamma
\end{equation*}
hold for $r>0$ and less than the diameter of $X$. The Hausdorff dimension with respect to $d$ of every ball in $X$ is $\gamma$. Replacing now the smooth test functions by compactly supported Lipschitz functions with respect to $d$, we have analogous for Lemmas~\ref{lem:VergenceFractDistributionCase1}, \ref{lem:VergenceFractBounded} and Proposition~\ref{propo:KirDivergenceFracCase1}, summarized in the next statement.
\begin{proposition}\label{propo:FractionalCase1LipschitzFunctions}
Let $0<s<\tfrac{1}{2}$, $(X,d,\mu)$ as before, $\mathscr{S}_1=Lip_0(X)$, the Lipschitz functions with bounded support in $X$, $\mathscr{S}_2= Lip_0(X\times X)$ the Lipschitz functions with bounded support in $X\times X$. Then, for $\Phi\in\mathscr{S}_2$,
\begin{enumerate}[(a)]
	\item the function $\frac{\Phi(x,y)-\Phi(x,x)}{d(x,y)^{\gamma +2s}}$ belongs to $L^1(X\times X, d\mu\times d\mu)$;
	\item the linear functional $S:\mathscr{S}_2\to\mathbb{R}$ given by 
	\begin{equation*}
	\langle\!\langle S,\Phi\rangle\!\rangle =\iint_{X\times X}\frac{\Phi(x,y)-\Phi(x,x)}{d(x,y)^{\gamma +2s}}d\mu(x) d\mu(y)
	\end{equation*}
	defines a distribution in $\mathscr{S}_2^{'}$;
	\item the function $\psi(x)=\int_X \frac{\Phi(x,y)-\Phi(x,x)}{d(x,y)^{\gamma+2s}} d\mu(y)$ belongs to $\mathscr{S}_1$;
	\item for $T=1$ and $S$ as in (b) we have
	\begin{equation*}
	Kir_{T,S} \Phi(x) = \int_X \frac{\Phi(x,y)-\Phi(x,x)}{d(x,y)^{\gamma+2s}} d\mu(y).
	\end{equation*}
\end{enumerate}
\end{proposition}
The proof follows the same lines of those in the first case in \S~\ref{sec:FractionalKirchhoffDivergencesEuclideanSpace}. A particular case which is interesting as a link between discrete and continuous cases is provided the dyadic settings introduced in \S~\ref{sec:DiracDeltasMetricMeasureSpaces}. Again, the situation could be introduced for very general families but the one dimensional case with the standard dyadic intervals provides all the ideas with a lower notational cost. We shall take in Proposition~\ref{propo:FractionalCase1LipschitzFunctions} $X=\mathbb{R}^+$, the set of nonnegative real numbers. Let $\mathcal{D}$ be the family of all dyadic intervals of $\mathbb{R}^+$, $\mathcal{D}=\bigcup_{j\in\mathbb{Z}}\mathcal{D}^j$, $\mathcal{D}^j=\{I^j_k: k=0,1,2,\ldots\}$, $I^j_k=[k2^{-j},(k+1)2^{-j})$. As in Section~\ref{sec:DiracDeltasMetricMeasureSpaces}, let $\rho(x,y)=\inf\{\abs{I}: I\in\mathcal{D} \textrm{ and } x,y\in I\}$. As it is easy to see $(\mathbb{R}^+,\rho,\abs{\cdot})$, with $\abs{\cdot}$ Lebesgue measure, is an Ahlfors space of dimension one. In fact, since $B_\rho(x,r)$ is the largest dyadic interval in $\mathbb{R}^+$ containing $x$ with length less than $r$, we have that for $j\in\mathbb{Z}$ such that $2^{j-1}<r\leq 2^j$, we have also that $2^{j-1}\leq \abs{B_\rho(x,r)}<2^j$. Hence $\tfrac{r}{2}\leq \abs{B_\rho(x,r)}\leq 2r$. Hence Proposition~\ref{propo:FractionalCase1LipschitzFunctions} can be applied in this space $(\mathbb{R}^+,\rho,\abs{\cdot})$. It is worthy noticing that the indicator functions of dyadic intervals are Lipschitz functions with respect to $\rho$ (see \cite{AiGoPetermichl}). Moreover, in \cite{AiBoGo13} it is shown that if $\Delta_s$ is the  $s$ Laplacian in this setting, i.e. if
\begin{equation*}
\Delta_s f(x) = \int_{\mathbb{R}^+}\frac{f(y) - f(x)}{\rho(x,y)^{1+2s}} dx dy,
\end{equation*}
then a complete system of eigenfunctions of $\Delta_s$ for $L^2(\mathbb{R}^+)$ is given by the Haar system. In other words with $c_s=\tfrac{2^{2s}}{2^{2s}-1}$ we have
\begin{equation*}
\Delta_s h = c_s \abs{\supp h}^{-2s} h,
\end{equation*}
for every $h\in\mathscr{H}=\{h^j_k(x)= 2^{j/2}h^0_0(2^j x-k): j\in\mathbb{Z}, k\geq 0\}$ with $h^0_0(x)=1$ in $[0,\tfrac{1}{2})$ and $h^0_0(x)=-1$ for $x\in [\tfrac{1}{2},1)$. This fact together with Proposition~\ref{propo:FractionalCase1LipschitzFunctions} give a formula for the Kirchhoff divergence in the dyadic setting which we state in the next result.

\begin{corollary}\label{coro:KirchhoffsFractionalHaarBasis}
Let $Kir_s \Phi$ be the Kirchhoff divergence operator provided by Proposition~\ref{propo:FractionalCase1LipschitzFunctions}
on the $1$-Ahlfors space $(\mathbb{R}^+,\rho,\abs{\cdot})$. Let $\mathscr{H}$ be the Haar basis of $L^2(\mathbb{R}^+)$. Then
\begin{equation*}
Kir_s \Phi(x) = c_s \sum_{h\in\mathscr{H}}\sum_{\widetilde{h}\in\mathscr{H}}\abs{\supp \widetilde{h}}^{-2s}\langle\!\langle \Phi, h\otimes\widetilde{h}\rangle\!\rangle h(x)\widetilde{h}(x);
\end{equation*}
where $(h\otimes\widetilde{h})(x,y)= h(x)\widetilde{h}(y)$ and
\begin{equation*}
\langle\!\langle \Phi, h\otimes\widetilde{h}\rangle\!\rangle = \iint_{\mathbb{R}^+\times\mathbb{R}^+}\Phi(y_1,y_2)h(y_1)\widetilde{h}(y_2) dy_1 dy_2
\end{equation*}
and $\Phi$ belongs to the linear span of the orthonormal basis of $L^2(\mathbb{R}^+\times\mathbb{R}^+)$ given by the tensor product $\mathscr{H}\otimes\mathscr{H}=\{h(x)\widetilde{h}(y): h, \widetilde{h}\in\mathscr{H}\}$.
\end{corollary}
\begin{proof}
Since $\Phi(x,y)=\sum_{h,\widetilde{h}\in\mathscr{H}} \langle\!\langle \Phi, h\otimes\widetilde{h}\rangle\!\rangle h(x)\widetilde{h}(y)$ and the sum is  finite, then
\begin{align*}
Kir_s \Phi (x) &= \int_{\mathbb{R}^+}\frac{\Phi(x,y)-\Phi(x,x)}{\rho(x,y)^{1+2s}} dy\\
&= \sum_{h, \widetilde{h}\in\mathscr{H}} \langle\!\langle \Phi, h\otimes\widetilde{h}\rangle\!\rangle h(x) \int_{y\in\mathbb{R}^+}\frac{\widetilde{h}(y)-\widetilde{h}(x)}{\rho(x,y)^{1+2s}}dy\\
&= \sum_{h, \widetilde{h}\in\mathscr{H}} \langle\!\langle \Phi, h\otimes\widetilde{h}\rangle\!\rangle h(x) \Delta_s \widetilde{h}(x)\\
&= c_s\sum_{h, \widetilde{h}\in\mathscr{H}} \langle\!\langle \Phi, h\otimes\widetilde{h}\rangle\!\rangle \abs{\supp\widetilde{h}}^{-2s}h(x) \widetilde{h}(x).
\end{align*}
Notice that as in Section~\ref{sec:ClassicalLaplacianRn} the above formula is a spectral version of $\Delta_{s,y}\Phi(x,x)$ and the underlying distribution $S$ in $\mathbb{R}^+\times \mathbb{R}^+$ is again $\Delta_{s,y}\mu$ where $\mu$ is the length in the diagonal.
\end{proof}

Let us finally observe that the results in \S~\ref{sec:FractionalKirchhoffDivergencesEuclideanSpace} and \S~\ref{sec:FractionalKirchhoffdivergenceAhlforsSpaces} can be extended to the more general kernels that have been considered as natural settings for some evolution equations of nonlinear variational type. See \cite{CaChaVa11} and \cite{CaSi14}, where the regularity theory of solutions of the diffusion associated to the Euler-Lagrange equation is considered. The generality of this type of kernels which do not need to be of convolution type, fits naturally in the general framework that we are considering. On the other hand, at least for the basic aspects of the theory, they have natural extensions to Ahlfors type metric spaces.

Let $X=\mathbb{R}^n$, $\mathscr{S}_1=\mathscr{C}_c(\mathbb{R}^n)$,  $\mathscr{S}_2=\mathscr{C}_c(\mathbb{R}^n\times \mathbb{R}^n)$. As in Section~\ref{sec:FractionalKirchhoffDivergencesEuclideanSpace} we shall consider $T=1$ and the master kernel $\mathcal{K}$ will define the distribution $S$ in $\mathscr{S}_2^{'}$.

 In \cite{CaChaVa11} the authors consider a symmetric kernel $\mathcal{K}$ defined on $\mathbb{R}^n\times\mathbb{R}^n$ such that for some $0<\sigma<1$ and for some positive constants $c_1\leq c_2$, satisfies the inequalities
\begin{equation}\label{eq:estimateskernelmasterequationtype}
c_1\mathcal{X}_{\{(x,y): \abs{x-y}<1\}}(x,y)\frac{1}{\abs{x-y}^{n+\sigma}}\leq \mathcal{K}(x,y)\leq\frac{c_2}{\abs{x-y}^{n+\sigma}}.
\end{equation}
With these estimates for the kernel $\mathcal{K}$ the arguments in the first case ($0<s<\tfrac{1}{2}$) in Section~\ref{sec:FractionalKirchhoffDivergencesEuclideanSpace} can be adapted to find a distribution $S_\sigma\in\mathscr{S}_2^{'}=\mathscr{C}_c^\infty(\mathbb{R}^{2n})$ such that $Kir_{T,S_\sigma}(f(y)-f(x))=\Delta_{T,S_\sigma}$ coincides with the operator
$\int_{\mathbb{R}^n}[f(y)-f(x)]\mathcal{K}(x,y) dy$
which is the Euler-Lagrange equation with quadratic energy $\iint_{\mathbb{R}^n\times\mathbb{R}^n}[f(x)-f(y)]^2\mathcal{K}(x,y) dx dy$.


The above situation extends naturally to Ahlfors regular metric spaces. With the notation of Section~\ref{sec:FractionalKirchhoffdivergenceAhlforsSpaces}, let $(X,d,\mu)$ be a $\gamma$-Ahlfors space. In this setting the upper bound in \eqref{eq:estimateskernelmasterequationtype} takes the form
\begin{equation*}\label{eq:AhlforsUpperBoundMasterEquation}
\mathcal{K}(x,y) \leq\frac{C}{d(x,y)^{\gamma+\sigma}},\quad x\neq y.
\end{equation*}
With similar arguments to those in Lemma~\ref{lem:VergenceFractDistributionCase1} and Proposition~\ref{propo:KirDivergenceFracCase1} we obtain the following result.
\begin{proposition}
Let $0<\sigma<1$, $(X,d,\mu)$ $\gamma$-Ahlfors, $\mathscr{S}_i$, $i=1,2$,  the Lipschitz functions with bounded support in $X$ and $X\times X$ respectively. Let $\mathcal{K}:X\times X\to\mathbb{R}$ be a nonnegative measurable kernel satisfying \eqref{eq:estimateskernelmasterequationtype}. Then, with $\proin{T}{\varphi}=\int_X\varphi d\mu$ and $\langle\!\langle S_\sigma,\Phi\rangle\!\rangle = \iint_{X\times X}[\Phi(x,y)-\Phi(x,x)] \mathcal{K}(x,y) d\mu(x)d\mu(y)$
we have
\begin{equation*}
Kir_{T,S_\sigma}\Phi(x) = \int [\Phi(x,y)-\Phi(x,x)]\mathcal{K}(x,y) d\mu(y).
\end{equation*}
\end{proposition}

\section{Some examples of \textit{(1.2.b)} with $T$ of positive order}\label{sec:SomeExamples12bTpositiveOrder}

So far we have considered examples of solutions of \textit{(1.2.b)} where the distributions $T$ are given by measures. In this section we aim to provide some examples with $T$ of positive order in the sense of distributions.

A simple Schwartz distribution in $\mathbb{R}$ which is neither a function nor a measure because it needs some positive regularity, aside of continuity, of the test functions is the principal value of $\tfrac{1}{x}$. The relevance of this distribution is that it is the kernel of the Hilbert transform. The paradigmatic singular integral operator.

Usually the distribution $p.v.\tfrac{1}{x}$ in $\mathscr{D}^{'}(\mathbb{R})$ is defined by
\begin{equation*}
\proin{p.v. \frac{1}{x}}{\varphi} = \lim_{\varepsilon\to 0}\int_{\abs{x}>\varepsilon}\frac{\varphi(x)}{x} dx,
\end{equation*}
for $\varphi\in\mathscr{C}^\infty_c(\mathbb{R})$. Since $p.v.\tfrac{1}{x}$ extends to $\mathscr{S}(\mathbb{R})$ the class of Schwartz of test functions it has a well defined Fourier transform which is a constant times the sign function on the frequency domain. The convolution of $p.v.\tfrac{1}{x}$ with a test function $\eta\in\mathscr{C}_c^\infty(\mathbb{R})$ is the Hilbert transform $H\eta (x)=\lim_{\varepsilon\to 0}\int_{\abs{x-y}>\varepsilon}\frac{\eta(y)}{x-y} dy$.

The very definition of $p.v.\tfrac{1}{x}$ allows to see this distribution as a limit, in the sense of distributions, of a sequence of functions. In fact, for each $\varepsilon>0$, $h_\varepsilon(x)=\tfrac{1}{x}\mathcal{X}_{\{\abs{x}>\varepsilon\}}(x)$ belongs to $L^1_{loc}(\mathbb{R})$ and hence to $\mathscr{D}^{'}$. Moreover $h_\varepsilon \overset{\mathscr{D}^{'}(\mathbb{R})}{\longrightarrow} p.v.\tfrac{1}{x}$. In our current situation the fact that each function $h_\varepsilon$ has a vanishing interval, namely $[-\varepsilon,\varepsilon]$, is not good for our division problem. The next elementary lemma gives a better adapted way of defining $p.v.\tfrac{1}{x}$.

\begin{lemma}
Set $h^\varepsilon = h_\varepsilon +\tfrac{\sigma}{\varepsilon}\mathcal{X}_{[-\varepsilon,\varepsilon]}$, where $\sigma$ is the sign function. Then $h^\varepsilon\in L^1_{loc}(\mathbb{R})$ and $h^\varepsilon\to p.v.\tfrac{1}{x}$ in $\mathscr{D}^{'}(\mathbb{R})$ as $\varepsilon\to 0$.
\end{lemma}
\begin{proof}
It is enough to show that $\tfrac{\sigma}{\varepsilon}\mathcal{X}_{[-\varepsilon,\varepsilon]}\overset{\mathscr{D}^{'}(\mathbb{R})}{\longrightarrow} 0$. Let $\varphi\in\mathscr{C}_c^\infty(\mathbb{R})$, then
\begin{equation*}
\int_{\mathbb{R}}\frac{\sigma(x)}{\varepsilon}\mathcal{X}_{[-\varepsilon,\varepsilon]}(x) \varphi(x) dx
= \frac{1}{\varepsilon}\int_0^\varepsilon \varphi(x) dx - \frac{1}{\varepsilon}\int_{-\varepsilon}^0 \varphi(x) dx
\end{equation*}
which converges to $\varphi(0)-\varphi(0)=0$ for $\varepsilon\to 0$, as desired.
\end{proof}

For positive $\varepsilon$, with the above notation, set, for $\Theta\in\mathscr{C}_c^\infty(\mathbb{R}^2)$
\begin{equation*}
\langle\!\langle S_\varepsilon,\Theta \rangle\!\rangle = \iint_{\mathbb{R}^2}\Phi(x,y)h_\varepsilon(x-y) dx dy =
\iint_{\abs{x-y}>\varepsilon}\Phi(x,y)\frac{1}{x-y} dx dy.
\end{equation*}
Since $h_\varepsilon(x-y)$ is bounded in $\mathbb{R}^2$, $S\varepsilon$ is well defined as a distribution in $\mathscr{D}^{'}(\mathbb{R}^2)$. So that equation \textit{(1.2.b)} with $T_\varepsilon$ the distribution in $\mathbb{R}$ induced by the locally integrable and non-vanishing function $h^\varepsilon$, and $S_\varepsilon$ defined above has the solution
\begin{equation}\label{eq:KirchhoffDivergenceHilbertTransformEpsilon}
Kir_\varepsilon \Phi (x) = \frac{1}{h^\varepsilon(x)}\int_{\{y:\abs{x-y}>\varepsilon\}}\Phi(x,y)\frac{1}{x-y} dy
\end{equation}
for every $\varepsilon>0$. Actually the above division of distributions is possible for $\varepsilon>0$ and the limit for $\varepsilon$ tending to zero is well defined.

\begin{proposition}
For a two variable function $\Theta(x,y)$ set $H_y\Theta(x,z)$ to denote the Hilbert transform of $\theta$ for fixed $x$, as a function of $y$, evaluated at $z$. Then
\begin{enumerate}[(i)]
\item $Kir_\varepsilon \Phi(x)\to x H_y\Phi(x,x)$, for $\varepsilon\to 0$ for every $x\in\mathbb{R}$;
\item for $T=p.v.\tfrac{1}{x}$ and $\langle\!\langle S,\Theta \rangle\!\rangle = \int_{\mathbb{R}} H_y\Theta (x,x) dx$ we have that $Kir_{T,S}\Phi(x)=x H_y\Phi(x,x)$.
\end{enumerate}
\end{proposition}
\begin{proof}
\textit{(i).} From equation \eqref{eq:KirchhoffDivergenceHilbertTransformEpsilon} and the definition of $h^\varepsilon$, we have
\begin{equation*}
Kir_\varepsilon \Phi(x) = \left(\varepsilon\sigma(x) \mathcal{X}_{\{\abs{x}\leq\varepsilon\}}(x) + x\mathcal{X}_{\{\abs{x}>\varepsilon\}}(x) \right) \int_{\{\abs{x-y}>\varepsilon\}}\Phi(x,y)\frac{dy}{x-y}.
\end{equation*}
Since the sections of $\Phi$ belong to $\mathscr{C}_c^\infty(\mathbb{R})$, we may take the limit for $\varepsilon$ going to zero to obtain \textit{(i)}.

\noindent\textit{(ii).} We have to check \textit{(1.2.b)} with $T$ and $S$ given in the statement. For $\varphi\in\mathscr{C}_c^\infty(\mathbb{R})$, we have
\begin{align*}
\proin{T}{\varphi(x) xH_y\Phi(x,x)} &= \lim_{\varepsilon\to 0}\int_{\abs{x}>\varepsilon}\frac{1}{x}[\varphi(x) x H_y\Phi(x,x)]dx\\
&=  \lim_{\varepsilon\to 0}\int_{\abs{x}>\varepsilon}\varphi(x) H_y\Phi(x,x) dx\\
&= \int_{\mathbb{R}}\varphi(x) H_y\Phi(x,x) dx\\
&= \int_{\mathbb{R}}H_y(\varphi\Phi)(x,x) dx\\
&= \langle\!\langle S,\varphi\Phi \rangle\!\rangle.
\end{align*}
Hence
\begin{equation*}
Kir_{T,S}\Phi(x) = x H_y\Phi(x,x),
\end{equation*}
as desired.
\end{proof}

The corresponding Laplacian type operator is given by
\begin{equation*}
\Delta_{T,S}f(x) = x Hf(x)
\end{equation*}
with the standard agreement of $H\,1=0$.

\section{Some convergence results}\label{sec:SomeConvergenceResults}

The issue of convergence of a sequence $Kir_k\Phi$; $k=0,1,2,\ldots$ associated to sequences $T_k$ and $S_k$ of distributions, presents several points of view  and several questions which could be of interest. Some of these aspects are classical. Such is the case of approximation of ``continuous'' operators by discrete operators. In this direction the example introduced in Proposition~\ref{propo:LaplacianOperatorFiniteDifference} of Section~\ref{sec:DiracDeltasEuclideanSpaces} is paradigmatic. Finite differences (graph structures) approximating the classical Laplacian on $\mathbb{R}^n$. Also the discrete fractional Laplacian in Proposition~\ref{propo:DiscreteFractionalLaplacian} can be viewed as a discrete approximation of the classical Laplacian. Less known and perhaps more difficult, but certainly more interesting looks the problem of searching the conditions on sequences of graphs such that the corresponding Kirchhoff divergences, and the corresponding Laplacians, converge to some
operator worthy of being considered a divergence or a Laplacian.

Let us first observe that the convergence of $T_k$ and $S_k$ in the sense of distributions of $\mathscr{S}_1^{'}$ and $\mathscr{S}_2^{'}$ respectively is not enough to have the convergence of $Kir_k = Kir_{T_k,S_k}$ as $k$ tends to infinity. In fact, take for example in $\mathscr{S}(\mathbb{R})$ the function $\eta(x)=\tfrac{1}{\sqrt{\pi}}e^{-x^2}$, the Gaussian function and $\eta_k(x)=k\eta(kx)$. Set $T_k$ to denote the distribution in $\mathscr{S}^{'}(\mathbb{R})$ induced by the integrable function $\eta_k$. Let $S=S_k$ be the Schwartz distribution in $\mathscr{S}^{'}(\mathbb{R}^2)$ induced by the area measure $dx dy$ in $\mathbb{R}^2$, that is $\langle\!\langle S_k,\Theta \rangle\!\rangle=\iint_{\mathbb{R}^2}\Theta(x,y) dx dy$, for every $k=0,1,2,\ldots$. Then, for $\Phi\in\mathscr{S}(\mathbb{R}^2)$ we have
\begin{equation*}
Kir_k \Phi(x) = \frac{1}{\eta_k(x)}\int_{y\in\mathbb{R}}\Phi(x,y) dy = \frac{\sqrt{\pi}}{k} e^{k^2x^2}\int_{y\in\mathbb{R}}\Phi(x,y) dy.
\end{equation*}
Which tends to zero when $x=0$ and when $x$ does not belong to the first projection of the support of $\Phi$. And for $\Phi\geq 0$, tends to $+\infty$ when $\int\Phi(x,y)dy$ is positive. In terms of Proposition~\ref{propo:IndependenceCaseKirchhoffLaplacian}, what happens in this example is that even when for each $k$ we have that $dx$ is absolutely continuous with respect to $\eta_k(x) dx$, this is no longer true for the limit since $\eta_k\to\delta_0$ as $k\to\infty$ in the sense of $\mathscr{S}^{'}(\mathbb{R})$.

On the other hand some simultaneous concentration of the measures defining $T_k$ and $S_k$ could allow the existence of a limit. In fact, take $T_k$ as before and $\langle\!\langle S_k,\Theta \rangle\!\rangle=\iint_{\mathbb{R}^2}\Theta(x,y)\zeta_k(x) dx dy$, with $\zeta_k(x)=k\zeta(kx)$ and $\zeta$ a probability density, i.e. $\int_{\mathbb{R}}\zeta dx=1$.

Now $Kir_k\Phi(x) = \frac{\zeta(kx)}{\eta(kx)}\int_{y\in\mathbb{R}}\Phi(x,y) dy$, and the limit for $k$ going to infinity depends on the relative size of the tails of $\zeta$ with respect to the Gaussian tails. Notice that $Kir_k \Phi(0)=\sqrt{\pi}\eta(0)\int_{y\in\mathbb{R}}\Phi(0,y)dy$ and also that if $\zeta$ has compact support then
\begin{equation*}
\lim_{k\to\infty}Kir_k\Phi(x) =
\begin{cases}
\sqrt{\pi}\eta(0)\int_{y\in\mathbb{R}}\Phi(0,y)dy, & x=0;\\
0, & \textrm{for } x\neq 0.
\end{cases}
\end{equation*}
On the other hand if $\zeta$ has heavy tails, like Cauchy distributions, then for $x\neq 0$, $Kir_k\Phi(x)$ tends to infinity, when $x$ belongs to the first projection of the support of $\Phi$ and $\Phi\geq 0$ and has positive integral.

With the notation of Section~\ref{sec:GeneralMeasures}, we have more interesting convergence cases when $T_k$ approaches $\delta_0$ in $\mathbb{R}$ with some specific rate and the measure $\Pi_k$ in $\mathbb{R}^2$ concentrates, with decreasing mass, about the diagonal of $\mathbb{R}^2$. Let us write an elementary case of this observation in the next statement.
\begin{proposition}\label{propo:LimitSequenceKirchhofftoDeltaDistribution}
Let $P(x)=\tfrac{1}{\pi}\tfrac{1}{1+x^2}$, $P_k(x)=kP(kx)$, $k=1,2,3,\ldots$ and $\proin{T_k}{\varphi}=\int_{\mathbb{R}}\varphi(x) P_k(x) dx$. Let $\Pi_k$ be the measure defined on the Borel sets of $\mathbb{R}^2$ by $\Pi_k(A)=\iint_A\mathcal{X}_{[-\tfrac{1}{k},\tfrac{1}{k}]}(x-y) dx dy$. Set $\langle\!\langle S_k,\Theta \rangle\!\rangle=\iint_{\mathbb{R}^2}\Theta d\Pi_k$. Then,
\begin{equation*}
\lim_{k\to\infty} Kir_k\Phi(x) = 2\pi x^2\Phi(x,x).
\end{equation*}
\end{proposition}
\begin{proof}
Let us first write out $Kir_k\Phi$ for $k=1,2,3,\ldots$ Since $P_k(x)$ never vanishes we have that
\begin{align*}
Kir_k \Phi(x) &= \frac{1}{P_k(x)}\int_{y\in\mathbb{R}}\Phi(x,y)\mathcal{X}_{[-\tfrac{1}{k},\tfrac{1}{k}]}(x-y) dy\\
&= \frac{\pi}{k}(1+k^2x^2)\int_{x-\tfrac{1}{k}}^{x+\tfrac{1}{k}}\Phi(x,y) dy\\
&= 2\pi\left(\frac{1}{k^2}+x^2\right)\frac{k}{2}\int_{x-\tfrac{1}{k}}^{x+\tfrac{1}{k}}\Phi(x,y) dy.
\end{align*}
The result follows taking the limit for $k\to\infty$.
\end{proof}

Notice that $S_k\to 0=S$ and $T_k\to \delta_0 = T$ in the sense of distributions and $Kir_{T,S}\Phi=0$ for every $\Phi$. Hence the operator that applies $\Phi(x,y)$ into $\tfrac{2}{\pi}x^2\Phi(x,x)$, obtained as a limit of this singular situation can be seen as  a generalization of the Kirchhoff divergence in this case. The formal Laplacian, instead, vanishes since $f(y)-f(x)$ is zero on the diagonal. On the other hand the exact value of the limit in the above proposition depends on the relative rates of convergence of $T_k$ to $\delta_0$ and of $S_k$ to $0$. If instead of the Cauchy density, or Poisson kernel, in Proposition~\ref{propo:LimitSequenceKirchhofftoDeltaDistribution} we use the Gauss kernel $\eta_k(x)=\tfrac{1}{\sqrt{\pi}}e^{-k^2 x^2}$, then a faster convergence of $S_k$ to zero is needed if we want to have a nontrivial limit.

The above considerations suggest that we are dealing with a notion of derivative depending on the rates of convergence of $S_k$ and $T_k$. Let us go back to a (continuous) parameter $h$ tending to zero instead of $\tfrac{1}{k}$ for $k\to\infty$.

In our general setting stated in Section~\ref{sec:intro}, let $\mathcal{T}:(-\varepsilon,\varepsilon)\to\mathscr{S}_1^{'}$ and $\mathcal{S}:(-\varepsilon,\varepsilon)\to\mathscr{S}_2^{'}$ be two distributions valued functions defined for each $h$ with $\abs{h}<\varepsilon$. Let $\Sigma(\Phi,h)$ be the distribution in $\mathscr{S}_1^{'}$ associated to $\Phi\in\mathscr{S}_2$ and $\mathcal{S}(h)$ by Lemma~\ref{lem:FunctionalSigmaonDistributions}. That is $\proin{\Sigma(\Phi,h)}{\varphi}= \langle\!\langle\mathcal{S}(h),\varphi\Phi\rangle\!\rangle$. We say that $\mathcal{S}$ is differentiable with respect to $\mathcal{T}$ at $\Phi\in\mathscr{S}_2$ if for some $\varepsilon>0$ the quotients $\mathcal{Q}(\Phi,h)=\frac{\Sigma(\Phi,h)}{\mathcal{T}(h)}$ are well defined as objects of $\mathscr{S}_1^{'}$ for $\abs{h}<\varepsilon$ and $\mathcal{Q}(\Phi,h)\to\tfrac{d\mathcal{S}}{d\mathcal{T}}(\Phi)\in\mathscr{S}_1^{'}$
in the sense of $\mathscr{S}_1^{'}$, as $h$ tends to zero. Some examples of existence and identification of these objects are in order.

\begin{theorem}
The basic setting is that of Proposition~\ref{propo:LaplacianOperatorFiniteDifference}, $\mathscr{S}_1=\mathscr{C}_c(\mathbb{R}^n)$, $\mathscr{S}_2=\mathscr{C}_c(\mathbb{R}^n\times\mathbb{R}^n)$. Set $\mathcal{T}:(0,1)\to\mathscr{S}_1^{'}$ given by $\mathcal{T}(h)=h^2\sum_{\vec{j}\in\mathbb{Z}^n} h^n\delta_{h\vec{j}}$. Set $\mathcal{S}:(0,1)\to \mathscr{S}_2^{'}$ given by $\mathcal{S}(h)=\sum_{\vec{k}\in\mathbb{Z}^n}\sum_{\{\vec{j}: \abs{\vec{j}-\vec{k}}=1\}} h^n\delta_{h\vec{k}}\times\delta_{h\vec{j}}$. Then, for $\Phi\in\mathscr{C}^2(\mathbb{R}^n\times\mathbb{R}^n)$ vanishing on the diagonal we have
\begin{equation*}
\frac{d\mathcal{S}}{d\mathcal{T}}(\Phi) = \Delta_y\Phi(x,x).
\end{equation*}
Moreover,
\begin{equation*}
\frac{d\mathcal{S}}{d\mathcal{T}}(f(y)-f(x)) = \Delta f,
\end{equation*}
for $f$ in $\mathscr{C}^2(\mathbb{R}^n)$.
\end{theorem}
\begin{proof}
We have that for positive $h$,
\begin{equation*}
\mathcal{Q}(\Phi,h) = \frac{\Sigma(\Phi,h)}{\mathcal{T}(h)} = Kir_h\Phi.
\end{equation*}
To check the convergence in $\mathscr{S}_1^{'}$ of $\mathcal{Q}(\Phi,h)$, with $\Phi\in\mathscr{C}^2(\mathbb{R}^{2n})$, take  a test function $\varphi\in\mathscr{S}_1=\mathscr{C}_c(\mathbb{R}^n)$, then from Proposition~\ref{propo:LaplacianOperatorFiniteDifference}, we have
\begin{align*}
&\proin{\mathcal{Q}(\Phi,h)}{\varphi} = \proin{Kir_{h}\Phi}{\varphi}\\
&= \sum_{\vec{k}\in\mathbb{Z}^n}\left(\frac{1}{h^2}\sum_{m=1}^n [\Phi(h\vec{k},h(\vec{k}+\vec{e_m})) + \Phi(h\vec{k},h(\vec{k}-\vec{e_m}))]\right)\int_{Q(h\vec{k})}\varphi(x) dx\\
&= \sum_{\vec{k}\in\mathbb{Z}^n}\left(\sum_{m=1}^n \frac{[\Phi(h\vec{k},h(\vec{k}+\vec{e_m})) - 2\Phi(h\vec{k},h\vec{k}) + \Phi(h\vec{k},h(\vec{k}-\vec{e_m}))]}{h^2}\right)\int_{Q(h\vec{k})}\varphi(x) dx.
\end{align*}
In the second equation above we are taking the continuous variable version of $Kir_h\Phi$ provided in Proposition~\ref{propo:LaplacianOperatorFiniteDifference}. In other words $Kir_h\Phi(x)=\sum_{\vec{k}\in\mathbb{Z}^n} Kir_h(h\vec{k})\mathcal{X}_{Q(h\vec{k})}(x)$, where $Q(h\vec{k})=\prod_{m=1}^n[h k_m, h(k_m+1)]$.
Since $\Phi$ is $\mathscr{C}^2(\mathbb{R}^2)$, from Taylor formula in the $y$ variables for $\Phi$ and letting $h$ tend to zero we get
\begin{equation*}
\proin{\mathcal{Q}(\Phi,h)}{\varphi} \to \int_{x\in\mathbb{R}^n}\Delta_y\Phi(x,x) \varphi(x) dx
\end{equation*}
as $h\to 0$ for every $\varphi\in\mathscr{C}_c(\mathbb{R}^n)$. Hence $\frac{d\mathcal{S}}{d\mathcal{T}}(\Phi)$ for $\Phi\in\mathscr{C}^2(\mathbb{R}^n\times\mathbb{R}^n)$ is the continuous function $\Delta_y\Phi(x,x)$.
\end{proof}

A second case of $\frac{d\mathcal{S}}{d\mathcal{T}}$ which has an explicit formula is the associated to Proposition~\ref{propo:DiscreteFractionalLaplacian}, \ref{propo:KirDivergenceFracCase1} and \ref{propo:KirFracRnIdenticallyOne} regarding fractional powers of the Laplacian.
\begin{theorem}
Let $\alpha>0$, $\mathscr{S}_1=\mathscr{C}_c(\mathbb{R}^n)$, $\mathscr{S}_2=\mathscr{C}_c(\mathbb{R}^n\times \mathbb{R}^n)$ and $\mathcal{T}:(0,1)\to\mathscr{S}_1^{'}$ given by $\mathcal{T}(h)=h^{\alpha}\sum_{\vec{k}\in\mathbb{Z}^n}h^n\delta_{h\vec{k}}$. Let $\mathcal{S}:(0,1)\to\mathscr{S}_2^{'}$ given by
\begin{equation*}
\mathcal{S}(h) = \sum_{\vec{k}\neq \vec{j}}h^{2n}\frac{1}{\abs{h\vec{k}-h\vec{j}}^{n+\alpha}}\delta_{(h\vec{k},h\vec{j})}.
\end{equation*}
Then for $0<\alpha<2$ and $\Phi\in\mathscr{S}_2$ vanishing on the diagonal and smooth we have
\begin{equation*}
\frac{\partial\mathcal{S}}{\partial\mathcal{T}}(\Phi) = (-\Delta)_y^{\tfrac{\alpha}{2}}\Phi(x,x).
\end{equation*}
\end{theorem}
\begin{proof}
For $h>0$ fixed we have that
\begin{equation*}
\mathcal{Q}(\Phi,h)=\frac{\Sigma(\Phi,h)}{\mathcal{T}(h)}(x)
=\sum_{\vec{k}\in\mathbb{Z}^n}\mathcal{X}_{Q(h\vec{k})}(x)
\left(\frac{1}{h^\alpha}\sum_{\vec{j}\neq\vec{k}}\frac{\Phi(h\vec{k},h\vec{j})}{\abs{h\vec{k}-h\vec{j}}^{n+\alpha}}\right),
\end{equation*}
from Proposition~\ref{propo:DiscreteFractionalLaplacian}. Take $\varphi\in\mathscr{S}_1=\mathscr{C}_c(\mathbb{R}^n)$, then
\begin{align*}
\proin{\mathcal{Q}(\Phi,h)}{\varphi} & =  \sum_{\vec{k}\in\mathbb{Z}^n}\left(\sum_{\vec{j}\neq\vec{k}}
\frac{\Phi(h\vec{k},h\vec{j})-\Phi(h\vec{k},h\vec{k})}{\abs{h\vec{k}-h\vec{j}}^{n+\alpha}} h^n\right)
\int_{Q(h\vec{k})}\varphi dx\\
&= \int_{\mathbb{R}^n}\varphi(x) \sum_{\vec{k}\in\mathbb{Z}^n}
\left(\sum_{\vec{j}\neq\vec{k}}
\frac{\Phi(h\vec{k},h\vec{j})-\Phi(h\vec{k},h\vec{k})}{\abs{h\vec{k}-h\vec{j}}^{n+\alpha}} h^n\right)
\mathcal{X}_{Q(h\vec{k})}(x) dx\\
&= \int_{\mathbb{R}^n}\varphi(x) \sum_{\vec{k}\in\mathbb{Z}^n}\sigma_k \mathcal{X}_{Q(h\vec{k})}(x) dx.
\end{align*}
For fixed $\vec{k}\in\mathbb{Z}^n$ and small $h>0$, let us divide the inner sum $\sigma_k$ in two parts: $\sigma_k^h=\sigma_k^{h1} + \sigma_k^{h2}$. Take
\begin{align*}
\sigma_k^{h1} &= \sum_{\{\vec{j}: 0<\abs{\vec{j}-\vec{k}}<\tfrac{1}{h}\}} \frac{\Phi(h\vec{k},h\vec{j})-\Phi(h\vec{k},h\vec{k})}{\abs{h\vec{k}-h\vec{j}}^{n+\alpha}} h^n\\
&= \sum_{\{\vec{j}: 0<\abs{\vec{j}-\vec{k}}<\tfrac{1}{h}\}} h^n \frac{\Phi(h\vec{k},h\vec{j})-\Phi(h\vec{k},h\vec{k})-\nabla_y \Phi(h\vec{k},h\vec{k})\cdot (\vec{j}-\vec{k})h}{\abs{h\vec{k}-h\vec{j}}^{n+\alpha}}\\
&\phantom{\sum} + \sum_{\{\vec{j}: 0<\abs{\vec{j}-\vec{k}}<\tfrac{1}{h}\}}
h^n \nabla_y \Phi(h\vec{k},h\vec{k})\cdot \frac{(h\vec{j}-h\vec{k})}{\abs{h\vec{k}-h\vec{j}}^{n+\alpha}}\\
&= \sigma_k^{h1,1} + \sigma_k^{h1,2}.
\end{align*}
Since $\Phi$ is of class $\mathscr{C}^2$ and has bounded support, $0<\alpha<2$ and $\varphi$ is continuous with compact support, then, the function of $x$ and $y$ given by
\begin{equation*}
\mathcal{X}_{\{\abs{x-y}<1\}}(x,y) \varphi(x) \frac{\Phi(x,y)-\Phi(x,x)-\nabla_y\Phi (x,x)\cdot (y-x)}{\abs{x-y}^{n+\alpha}}
\end{equation*}
is absolutely integrable in $\mathbb{R}^{2n}$. See the arguments in Lemma~\ref{lemma:VergenceFracDistributionCase2} above. Moreover the above function of $(x,y)\in\mathbb{R}^{2n}$ is continuous except on the diagonal $x=y$ and on $\abs{x-y}=1$. Hence we can approximate its double integral on $\mathbb{R}^{2n}$ through Riemann sums. So that
\begin{align*}
&\int_{\mathbb{R}^n} \varphi(x)\int_{\abs{x-y}<1} \frac{\Phi(x,y)-\Phi(x,x)-\nabla_y \Phi(x,x)\cdot (y-x)}{\abs{x-y}^{n+\alpha}} dy dx\\
&= \lim_{h\to 0^+}\sum_{\vec{k}\in\mathbb{Z}^n} \varphi(h\vec{k})\sum_{0<\abs{h\vec{k}-h\vec{j}}<1}
\frac{\Phi(h\vec{k},h\vec{j})-\Phi(h\vec{k},h\vec{k})-\nabla_y \Phi(h\vec{k},h\vec{k})\cdot (\vec{j}-\vec{k})h}{\abs{h\vec{k}-h\vec{j}}^{n+\alpha}} h^{2n}\\
&= \lim_{h\to 0^+}\int_{\mathbb{R}^n} \varphi(x)\sum_{\vec{k}\in\mathbb{Z}^n}\mathcal{X}_{Q(h\vec{k})}(x)\\
&\phantom{\lim_{h\to 0^+}\int_{\mathbb{R}^n}}
\cdot\left(\sum_{0<\abs{\vec{j}-\vec{k}}<\frac{1}{h}} h^n \frac{\Phi(h\vec{k},h\vec{j})-\Phi(h\vec{k},h\vec{k})-\nabla_y \Phi(h\vec{k},h\vec{k})\cdot (\vec{j}-\vec{k})h}{\abs{h\vec{k}-h\vec{j}}^{n+\alpha}}\right) dx\\
&=\lim_{h\to 0^+}\int_{\mathbb{R}^n} \varphi(x)\sum_{\vec{k}\in\mathbb{Z}^n}\sigma_k^{h1,1}\mathcal{X}_{Q(h\vec{k})}(x) dx.
\end{align*}
Hence, since $\Phi(x,x)=0$ and since the integral on a ball centered at $x$ of radius $\frac{y_i-x_i}{\abs{x-y}^{n+\alpha}}$ vanishes for all $i=1,\ldots,n$, we have
\begin{align*}
&\lim_{h\to 0^+}\int_{\mathbb{R}^n} \varphi(x)\sum_{\vec{k}\in\mathbb{Z}^n}\sigma_k^{h1,1}\mathcal{X}_{Q(h\vec{k})}(x) dx\\
&= \int_{x\in\mathbb{R}^n} \varphi(x)\int_{\abs{x-y}<1}\frac{\Phi(h\vec{k},h\vec{j})-\Phi(h\vec{k},h\vec{k})-\nabla_y \Phi(h\vec{k},h\vec{k})\cdot (\vec{j}-\vec{k})h}{\abs{h\vec{k}-h\vec{j}}^{n+\alpha}} dy dx\\
&= \int_{x\in\mathbb{R}^n} \varphi(x)\left(\int_{\abs{x-y}<1} \frac{\Phi(x,y)}{\abs{x-y}^{n+\alpha}} dy\right) dx.
\end{align*}
Let us consider the term $\sigma_k^{h1,2}$. Notice that
\begin{align*}
\sigma_k^{h1,2} &= h^n \nabla_y\Phi(h\vec{k},h\vec{k})\cdot \sum_{0<\abs{\vec{j}-\vec{k}}<\tfrac{1}{h}} \frac{h\vec{j}-h\vec{k}}{\abs{h\vec{j}-h\vec{k}}^{n+\alpha}}\\
&= h^{1-\alpha}\nabla_y\Phi(h\vec{k},h\vec{k})\cdot \sum_{0<|\vec{i}|<\tfrac{1}{h}} \frac{\vec{i}}{|\vec{i}|^{n+\alpha}}.
\end{align*}
Since for $\vec{i}\in\mathbb{Z}^n$ with $0<|\vec{i}|<\tfrac{1}{h}$ we have that $-\vec{i}$ satisfies the same condition, we have that the last sum is the zero vector in $\mathbb{R}^n$ and $\sigma_k^{h1,2}=0$, for every $\vec{k}$ and for every $h>0$.

Let us now consider the convergence for $\sigma_k^{h2}$. We have to prove that
\begin{align*}
&\lim_{h\to 0^+}\int_{\mathbb{R}^n} \varphi(x)\sum_{\vec{k}\in\mathbb{Z}^n}\mathcal{X}_{Q(h\vec{k})}
\left(\sum_{\abs{\vec{j}-\vec{k}}>\tfrac{1}{h}} h^n \frac{\Phi(h\vec{k},h\vec{j})}{\abs{h\vec{k}-h\vec{j}}^{n+\alpha}}\right) dx\\
&=\int_{\mathbb{R}^n} \varphi(x)\int_{\{y: \abs{x-y}\geq 1\}}\frac{\Phi(x,y)}{\abs{x-y}^{n+\alpha}} dy dx.
\end{align*}
But this fact is again a consequence of the continuity and support properties of $\Phi$ and $\varphi$. Hence we have
\begin{align*}
\lim_{h\to 0^+}\proin{\mathcal{Q}(\Phi,h)}{\varphi} &= \int_{\mathbb{R}^n}\varphi(x) \left(\int_{y\in\mathbb{R}^n}\frac{\Phi(x,y)}{\abs{x-y}^{n+\alpha}} dy\right)dx\\
&= \int_{\mathbb{R}^n}\varphi(x)(-\Delta^\alpha)_y\Phi(x,x) dx,
\end{align*}
as desired.
\end{proof}

 Even when the above results are reformulations of known methods of approximation by finite differences of integer and fractional differential operators, the existence of $\frac{d\mathcal{S}}{d\mathcal{T}}$ for the functions $\mathcal{S}(h)$ and $\mathcal{T}(h)$ given by graphs, could be of help at understanding the diffusion processes on such structures.

As a final example let us compute the derivative $\frac{d\mathcal{S}}{d\mathcal{T}}$ for some deterministic coupled measures. For the sake of simplicity we choose a very elementary case of a certainly more general situation.

\begin{proposition}
Let $X=[0,1]$ with $\mathscr{S}_1$ and $\mathscr{S}_2$ the spaces of continuous functions on $X$ and $X\times X$ respectively. Let $F=F(h,x)$ be, for each $h\in [0,1]$, a function from $[0,1]$ into itself. Assume that $F$ is differentiable and that $F(0,x)=x$, the identity on $[0,1]$. Let $\mathcal{T}(h)$ be the measure $\mu$ with $d\mu=dx$ defined on $[0,1]$. Let $\pi_h$ be the measure defined on the Borel sets of $[0,1]^2$ by $\pi_h=\mu\circ G^{-1}_h$ with $G_h:[0,1]\to [0,1]^2$, is given by $G_h(x)=(x,F(h,x))$. In other words $\pi_h(A)=\mu(\{x: G_h(x)\in A\})$. Set $\mathcal{S}(h)$ to denote the distribution induced by $\pi_h$ on $\mathscr{S}_2$. Let $\Phi\in\mathscr{S}_2$ be a smooth function vanishing on the diagonal. Then $\frac{d\mathcal{S}}{d\mathcal{T}}(\Phi)$ exists and is the function in $[0,1]$ given by
\begin{equation*}
\frac{d\mathcal{S}}{d\mathcal{T}}(\Phi)(x) = \frac{\partial F}{\partial h}(0,x) \frac{\partial\Phi}{\partial y}(x,x).
\end{equation*}
\end{proposition}
\begin{proof}
For $h>0$ fixed we have
\begin{align*}
h\int_{[0,1]} \varphi(x) \psi_h(x) dx &= \proin{\mathcal{T}(h)}{\varphi\psi_h}\\
&= \langle\!\langle \mathcal{S}(h),\varphi\Phi\rangle\!\rangle\\
&= \iint_{[0,1]^2}\varphi(x)\Phi(x,y) d\pi_h(x,y)\\
&= \int_{[0,1]}\varphi(x)\Phi(x,F(h,x)) dx,
\end{align*}
for every $\varphi\in\mathscr{C}([0,1])$. Then, $\mathcal{Q}(\Phi,h)=\psi_h$ is the function
\begin{align*}
\mathcal{Q}(\Phi,h) &= \frac{1}{h}\Phi(x,F(h,x))\\
&= \frac{1}{h} [\Phi(x,F(h,x))-\Phi(x,x)]\\
&= \frac{\Phi(x,F(h,x))-\Phi(x,F(0,x))}{h}.
\end{align*}
And the result follows by taking $\lim_{h\to 0}\mathcal{Q}(\Phi,h)$.
\end{proof}

An example of the above is provided by the approximation $F(h,x)=x^{1+h}$ of the diagonal of $[0,1]^2$. In this case
\begin{equation*}
\frac{d\mathcal{S}}{d\mathcal{T}}(\Phi)(x) = x\log x \frac{\partial\Phi}{\partial y}(x,x).
\end{equation*}


\newcommand{\etalchar}[1]{$^{#1}$}
\providecommand{\bysame}{\leavevmode\hbox to3em{\hrulefill}\thinspace}
\providecommand{\MR}{\relax\ifhmode\unskip\space\fi MR }
\providecommand{\MRhref}[2]{%
	\href{http://www.ams.org/mathscinet-getitem?mr=#1}{#2}
}
\providecommand{\href}[2]{#2}



\bigskip

\noindent{\footnotesize Hugo Aimar, and Ivana G\'omez.
	\textsc{Instituto de Matem\'{a}tica Aplicada del Litoral, CONICET, UNL.}
	
	%
%
\medskip

%

\noindent \textit{Address.} \textmd{IMAL, CCT CONICET Santa Fe, Predio ``Alberto Cassano'', Colectora Ruta Nac.~168 km~0, Paraje El Pozo, S3007ABA Santa Fe, Argentina.}
}

\bigskip

\end{document}